\documentclass[12pt]{article}
\usepackage{a4}
\usepackage{amsthm}
\usepackage{amsmath}
\usepackage{amsfonts}
\usepackage{amssymb}
\usepackage{stmaryrd}
\usepackage{cite}
\usepackage{epsfig}
\usepackage[skip=0pt]{caption}
\usepackage[T1]{fontenc}
\newtheorem{theorem}{Theorem}
\newtheorem{proposition}[theorem]{Proposition}
\newtheorem{lemma}[theorem]{Lemma}

\newcommand{\FF}{{\cal F}}
\newcommand{\EE}{{\mathbb E}}
\newcommand{\NN}{{\mathbb N}}
\newcommand{\ZZ}{{\mathbb Z}}
\newcommand{\RR}{{\mathbb R}}
\newcommand{\unlabel}[2]{\left\llbracket #1\right\rrbracket_{#2}}
\newcommand{\flageroot}{\,{\mathchoice
    {\epsfysize 1.92ex \raisebox{-.3ex}{\epsfbox{profile3.101}}}
    {\epsfysize 1.92ex \raisebox{-.3ex}{\epsfbox{profile3.101}}}
    {\epsfysize 1.44ex \epsfbox{profile3.101}}
    {\epsfysize 1.08ex \epsfbox{profile3.101}}
  }\,}
\newcommand{\flagfroot}{\,{\mathchoice
    {\epsfysize 1.92ex \raisebox{-.3ex}{\epsfbox{profile3.102}}}
    {\epsfysize 1.92ex \raisebox{-.3ex}{\epsfbox{profile3.102}}}
    {\epsfysize 1.44ex \epsfbox{profile3.102}}
    {\epsfysize 1.08ex \epsfbox{profile3.102}}
  }\,}
\newcommand{\flagHzero}{\,{\mathchoice
    {\epsfysize 1.92ex \raisebox{-.3ex}{\epsfbox{profile3.103}}}
    {\epsfysize 1.92ex \raisebox{-.3ex}{\epsfbox{profile3.103}}}
    {\epsfysize 1.44ex \epsfbox{profile3.103}}
    {\epsfysize 1.08ex \epsfbox{profile3.103}}
  }\,}
\newcommand{\flagHone}{\,{\mathchoice
    {\epsfysize 1.92ex \raisebox{-.3ex}{\epsfbox{profile3.104}}}
    {\epsfysize 1.92ex \raisebox{-.3ex}{\epsfbox{profile3.104}}}
    {\epsfysize 1.44ex \epsfbox{profile3.104}}
    {\epsfysize 1.08ex \epsfbox{profile3.104}}
  }\,}
\newcommand{\flagHtwo}{\,{\mathchoice
    {\epsfysize 1.92ex \raisebox{-.3ex}{\epsfbox{profile3.105}}}
    {\epsfysize 1.92ex \raisebox{-.3ex}{\epsfbox{profile3.105}}}
    {\epsfysize 1.44ex \epsfbox{profile3.105}}
    {\epsfysize 1.08ex \epsfbox{profile3.105}}
  }\,}
\newcommand{\flagHthree}{\,{\mathchoice
    {\epsfysize 1.92ex \raisebox{-.3ex}{\epsfbox{profile3.106}}}
    {\epsfysize 1.92ex \raisebox{-.3ex}{\epsfbox{profile3.106}}}
    {\epsfysize 1.44ex \epsfbox{profile3.106}}
    {\epsfysize 1.08ex \epsfbox{profile3.106}}
  }\,}
\newcommand{\flagErootzero}{\,{\mathchoice
    {\epsfysize 1.92ex \raisebox{-.3ex}{\epsfbox{profile3.107}}}
    {\epsfysize 1.92ex \raisebox{-.3ex}{\epsfbox{profile3.107}}}
    {\epsfysize 1.44ex \epsfbox{profile3.107}}
    {\epsfysize 1.08ex \epsfbox{profile3.107}}
  }\,}
\newcommand{\flagErootone}{\,{\mathchoice
    {\epsfysize 1.92ex \raisebox{-.3ex}{\epsfbox{profile3.108}}}
    {\epsfysize 1.92ex \raisebox{-.3ex}{\epsfbox{profile3.108}}}
    {\epsfysize 1.44ex \epsfbox{profile3.108}}
    {\epsfysize 1.08ex \epsfbox{profile3.108}}
  }\,}
\newcommand{\flagEroottwo}{\,{\mathchoice
    {\epsfysize 1.92ex \raisebox{-.3ex}{\epsfbox{profile3.109}}}
    {\epsfysize 1.92ex \raisebox{-.3ex}{\epsfbox{profile3.109}}}
    {\epsfysize 1.44ex \epsfbox{profile3.109}}
    {\epsfysize 1.08ex \epsfbox{profile3.109}}
  }\,}
\newcommand{\flagErootboth}{\,{\mathchoice
    {\epsfysize 1.92ex \raisebox{-.3ex}{\epsfbox{profile3.110}}}
    {\epsfysize 1.92ex \raisebox{-.3ex}{\epsfbox{profile3.110}}}
    {\epsfysize 1.44ex \epsfbox{profile3.110}}
    {\epsfysize 1.08ex \epsfbox{profile3.110}}
  }\,}
\newcommand{\flagEroot}{\,{\mathchoice
    {\epsfysize 1.92ex \raisebox{-.3ex}{\epsfbox{profile3.111}}}
    {\epsfysize 1.92ex \raisebox{-.3ex}{\epsfbox{profile3.111}}}
    {\epsfysize 1.44ex \epsfbox{profile3.111}}
    {\epsfysize 1.08ex \epsfbox{profile3.111}}
  }\,}
\newcommand{\flagFrootzero}{\,{\mathchoice
    {\epsfysize 1.92ex \raisebox{-.3ex}{\epsfbox{profile3.112}}}
    {\epsfysize 1.92ex \raisebox{-.3ex}{\epsfbox{profile3.112}}}
    {\epsfysize 1.44ex \epsfbox{profile3.112}}
    {\epsfysize 1.08ex \epsfbox{profile3.112}}
  }\,}
\newcommand{\flagFrootone}{\,{\mathchoice
    {\epsfysize 1.92ex \raisebox{-.3ex}{\epsfbox{profile3.113}}}
    {\epsfysize 1.92ex \raisebox{-.3ex}{\epsfbox{profile3.113}}}
    {\epsfysize 1.44ex \epsfbox{profile3.113}}
    {\epsfysize 1.08ex \epsfbox{profile3.113}}
  }\,}
\newcommand{\flagFroottwo}{\,{\mathchoice
    {\epsfysize 1.92ex \raisebox{-.3ex}{\epsfbox{profile3.114}}}
    {\epsfysize 1.92ex \raisebox{-.3ex}{\epsfbox{profile3.114}}}
    {\epsfysize 1.44ex \epsfbox{profile3.114}}
    {\epsfysize 1.08ex \epsfbox{profile3.114}}
  }\,}
\newcommand{\flagFrootboth}{\,{\mathchoice
    {\epsfysize 1.92ex \raisebox{-.3ex}{\epsfbox{profile3.115}}}
    {\epsfysize 1.92ex \raisebox{-.3ex}{\epsfbox{profile3.115}}}
    {\epsfysize 1.44ex \epsfbox{profile3.115}}
    {\epsfysize 1.08ex \epsfbox{profile3.115}}
  }\,}
\newcommand{\flagFroot}{\,{\mathchoice
    {\epsfysize 1.92ex \raisebox{-.3ex}{\epsfbox{profile3.116}}}
    {\epsfysize 1.92ex \raisebox{-.3ex}{\epsfbox{profile3.116}}}
    {\epsfysize 1.44ex \epsfbox{profile3.116}}
    {\epsfysize 1.08ex \epsfbox{profile3.116}}
  }\,}
\newcommand{\flagfourtwo}{\,{\mathchoice
    {\epsfysize 1.92ex \raisebox{-.3ex}{\epsfbox{profile3.117}}}
    {\epsfysize 1.92ex \raisebox{-.3ex}{\epsfbox{profile3.117}}}
    {\epsfysize 1.44ex \epsfbox{profile3.117}}
    {\epsfysize 1.08ex \epsfbox{profile3.117}}
  }\,}
\newcommand{\flagfourthree}{\,{\mathchoice
    {\epsfysize 1.92ex \raisebox{-.3ex}{\epsfbox{profile3.118}}}
    {\epsfysize 1.92ex \raisebox{-.3ex}{\epsfbox{profile3.118}}}
    {\epsfysize 1.44ex \epsfbox{profile3.118}}
    {\epsfysize 1.08ex \epsfbox{profile3.118}}
  }\,}
\newcommand{\flagfoursix}{\,{\mathchoice
    {\epsfysize 1.92ex \raisebox{-.3ex}{\epsfbox{profile3.119}}}
    {\epsfysize 1.92ex \raisebox{-.3ex}{\epsfbox{profile3.119}}}
    {\epsfysize 1.44ex \epsfbox{profile3.119}}
    {\epsfysize 1.08ex \epsfbox{profile3.119}}
  }\,}
\newcommand{\flagfourseven}{\,{\mathchoice
    {\epsfysize 1.92ex \raisebox{-.3ex}{\epsfbox{profile3.120}}}
    {\epsfysize 1.92ex \raisebox{-.3ex}{\epsfbox{profile3.120}}}
    {\epsfysize 1.44ex \epsfbox{profile3.120}}
    {\epsfysize 1.08ex \epsfbox{profile3.120}}
  }\,}
\newcommand{\flagfoureight}{\,{\mathchoice
    {\epsfysize 1.92ex \raisebox{-.3ex}{\epsfbox{profile3.121}}}
    {\epsfysize 1.92ex \raisebox{-.3ex}{\epsfbox{profile3.121}}}
    {\epsfysize 1.44ex \epsfbox{profile3.121}}
    {\epsfysize 1.08ex \epsfbox{profile3.121}}
  }\,}
\newcommand{\flagfournine}{\,{\mathchoice
    {\epsfysize 1.92ex \raisebox{-.3ex}{\epsfbox{profile3.122}}}
    {\epsfysize 1.92ex \raisebox{-.3ex}{\epsfbox{profile3.122}}}
    {\epsfysize 1.44ex \epsfbox{profile3.122}}
    {\epsfysize 1.08ex \epsfbox{profile3.122}}
  }\,}
\newcommand{\flagcherryonethree}{\,{\mathchoice
    {\epsfysize 1.92ex \raisebox{-.3ex}{\epsfbox{profile3.123}}}
    {\epsfysize 1.92ex \raisebox{-.3ex}{\epsfbox{profile3.123}}}
    {\epsfysize 1.44ex \epsfbox{profile3.123}}
    {\epsfysize 1.08ex \epsfbox{profile3.123}}
  }\,}
\newcommand{\flagcherrytwothree}{\,{\mathchoice
    {\epsfysize 1.92ex \raisebox{-.3ex}{\epsfbox{profile3.124}}}
    {\epsfysize 1.92ex \raisebox{-.3ex}{\epsfbox{profile3.124}}}
    {\epsfysize 1.44ex \epsfbox{profile3.124}}
    {\epsfysize 1.08ex \epsfbox{profile3.124}}
  }\,}
\newcommand{\flagcherrytwo}{\,{\mathchoice
    {\epsfysize 1.92ex \raisebox{-.3ex}{\epsfbox{profile3.125}}}
    {\epsfysize 1.92ex \raisebox{-.3ex}{\epsfbox{profile3.125}}}
    {\epsfysize 1.44ex \epsfbox{profile3.125}}
    {\epsfysize 1.08ex \epsfbox{profile3.125}}
  }\,}
\newcommand{\flagcherryone}{\,{\mathchoice
    {\epsfysize 1.92ex \raisebox{-.3ex}{\epsfbox{profile3.126}}}
    {\epsfysize 1.92ex \raisebox{-.3ex}{\epsfbox{profile3.126}}}
    {\epsfysize 1.44ex \epsfbox{profile3.126}}
    {\epsfysize 1.08ex \epsfbox{profile3.126}}
  }\,}
\newcommand{\flage}{\,{\mathchoice
    {\epsfysize 1.92ex \raisebox{-.3ex}{\epsfbox{profile3.127}}}
    {\epsfysize 1.92ex \raisebox{-.3ex}{\epsfbox{profile3.127}}}
    {\epsfysize 1.44ex \epsfbox{profile3.127}}
    {\epsfysize 1.08ex \epsfbox{profile3.127}}
  }\,}

\begin{document}
\title{Densities of 3-vertex graphs\thanks{The work of the first author was supported by the ERC grant ``High-dimensional combinatorics'' at the Hebrew University. The work of the third and fifth authors was supported by the European Research Council (ERC) under the European Union's Horizon 2020 research and innovation programme (grant agreement No 648509). This publication reflects only its authors' view; the European Research Council Executive Agency is not responsible for any use that may be made of the information it contains. The second author was partially supported by the NCN grant 2013/08/T/ST1/00108. The work of the fourth and fifth authors was supported by the Engineering and Physical Sciences Research Council Standard Grant number EP/M025365/1. The sixth author was partially supported by the SNSF grant 200021-149111 and by CRM-ISM fellowship.}}
\author{Roman Glebov\thanks{School of Computer Science and Engineering, Hebrew University, Jerusalem 9190401, Israel. E-mail: {\tt roman.l.glebov@gmail.com}.}\and
        Andrzej Grzesik\thanks{Faculty of Mathematics and Computer Science, Jagiellonian University, \L ojasiewicza 6, 30-348 Krak\'ow, Poland. E-mail: {\tt Andrzej.Grzesik@tcs.uj.edu.pl}.}\and
	Ping Hu\thanks{Department of Computer Science and DIMAP, University of Warwick, Coventry CV4 7AL, UK. E-mail: {\tt \{p.hu,t.hubai\}@warwick.ac.uk}.}\and
\newcounter{lth}
\setcounter{lth}{4}
	Tam\'as Hubai$^\fnsymbol{lth}$\and
        Daniel Kr\'al'\thanks{Mathematics Institute, DIMAP and Department of Computer Science, University of Warwick, Coventry CV4 7AL, UK. E-mail: {\tt d.kral@warwick.ac.uk}.}\and
	Jan Volec\thanks{Department of Mathematics and Statistics, McGill University, Burnside Hall, 805 Sherbrooke West, Montreal H3A 2K6, Canada. Previous affiliation: Department of Mathematics, ETH, 8092 Z\"urich, Switzerland. E-mail: {\tt jan@ucw.cz}.}}
\date{}
\maketitle
\begin{abstract}
Let $d_i(G)$ be the density of the $3$-vertex $i$-edge graph in a graph $G$,
i.e., the probability that three random vertices induce a subgraph with $i$ edges.
Let $S\subseteq\RR^4$ be the set of all quadruples $(d_0,d_1,d_2,d_3)$ that
are arbitrary close to $3$-vertex graph densities in arbitrary large graphs.
Huang, Linial, Naves, Peled and Sudakov have recently
determined the projection of the set $S$ to the $(d_0,d_3)$ plane.
We determine the projection of the set $S$ to all the remaining planes.
\end{abstract}

\section{Introduction}

Many problems in graph theory relate to understanding possible combinations of subgraph densities in graphs.
Indeed, the study of possible subgraph densities forms a very important area of extremal graph theory,
which contains many classical results but which is also full of hard and challenging problems.
The classical results include,
e.g., Tur\'an's Theorem~\cite{bib-turan41} determining the maximum edge density in $K_r$-free graphs,
Goodman's Bound~\cite{bib-goodman59} relating the densities of $\overline{K_3}$ and $K_3$, and
Kruskal-Katona Theorem~\cite{bib-katona68,bib-kruskal63}.
On the other hand, one of the recent breakthroughs in extremal graph theory was
the description of possible densities of complete graphs in graphs with a given edge density,
which is given in the exciting work of Razborov~\cite{bib-flag11},
Nikiforov~\cite{bib-nikiforov11} and Reiher~\cite{bib-reiher16}.

While problems related to possible densities of small graphs may look innocent at the first sight, they can become incredibly challenging.
For example, determining the minimum possible sum of densities of $\overline{K_4}$ and $K_4$ is a well-known problem in graph theory,
which is open for more than five decades.
Erd\H os~\cite{bib-erdos62} conjectured this minimum to be $1/32$,
which was dispoved by Thomasson~\cite{bib-thomason89}, who constructed graphs with the sum of the two densities below $1/32$.
However, despite extensive subsequent work on the problem, e.g.,~\cite{bib-franek93+,bib-sperfeld,bib-thomason97,bib-wolf10},
there is not even a construction that is believed to provide the tight bound for this problem.
Hence, determining all possible densities of $\overline{K_4}$ and $K_4$ look completely hopeless.
On the other hand, Huang, Linial, Naves, Peled and Sudakov~\cite{bib-huang14+},
building on their results from~\cite{bib-huang+}, determined possible densities of $\overline{K_3}$ and $K_3$.
We contribute to this line of research by completing the description of possible densities of all pairs of 3-vertex graphs.
Our results are encouraging to make an attempt to describe all possible combinations of 3-vertex graph densities,
which would imply the earlier mentioned result of Razborov~\cite{bib-flag11}
on the minimum triangle density in a graph with a given edge density.

To state our results precisely, we need several definitions.
The density $d(H,G)$ of a $k$-vertex graph $H$ in a graph $G$ is the probability that
$k$ randomly chosen vertices of $G$ induce a subgraph isomorphic to~$H$.
We will be interested in densities of $3$-vertex graphs.
There are four $3$-vertex graphs:
the triangle $K_3$, the cherry $K_{1,2}$, the co-cherry $\overline{K_{1,2}}$ and the co-triangle $\overline{K_3}$;
let $H_k$, $k=0,1,2,3$, be the $3$-vertex graph with $k$ edges.
Further, let $S\subseteq\RR^4$ be the set of all quadruples $(d_0,d_1,d_2,d_3)$ such that
for every $\varepsilon>0$ and every $n\in\NN$,
there exists a graph $G$ with at least $n$ vertices such that
the density of $H_k$ in $G$ differs from $d_k$ by at most $\varepsilon$ for $k=0,1,2,3$.
This is equivalent to saying that $(d_0,d_1,d_2,d_3)\in S$ if and only if
there exists a sequence of graphs such that their number of vertices tends to infinity and
the density of $H_k$ in the graphs forming the sequence converges to $d_k$ for every $k=0,1,2,3$.
In~\cite{bib-huang14+},
the set $S$ is referred to as the set of $3$-local profiles of arbitrary large graphs.
There is also an alternative description of the set $S$ using the theory of graph limits,
which we present in Section~\ref{sec-limits}.

Let $S_{ij}$ be the projection of the set $S$ to the plane of the $i$-th and $j$-th coordinate.
Huang, Linial, Naves, Peled and Sudakov~\cite{bib-huang14+},
building on their results from~\cite{bib-huang+}, determined the projection $S_{03}$ of the set $S$.
In particular, they showed that $(d_0,d_3)\in S_{03}$ if and only if
$d_0\ge 0$, $d_3\ge 0$, $d_0+d_3\ge 1/4$ (this inequality is equivalent to Goodman's bound), and
$$d_3\le\max\{(1-d_0^{1/3})^3+3d_0^{1/3}(1-d_0^{1/3})^2,(1-\alpha)^3\}$$
where $\alpha$ is the unique root in $[0,1]$ of the equation $\alpha^3+3\alpha^2(1-\alpha)=d_0$.
The set $S_{03}$ is visualized in Figure~\ref{fig-tr-ct}.
The upper curve in Figure~\ref{fig-tr-ct} corresponds to densities of co-triangles and triangles
in the graph consisting of a complete graph and isolated vertices or
in the complement of this graph.

\begin{figure}[ht]
\begin{center}
\epsfbox{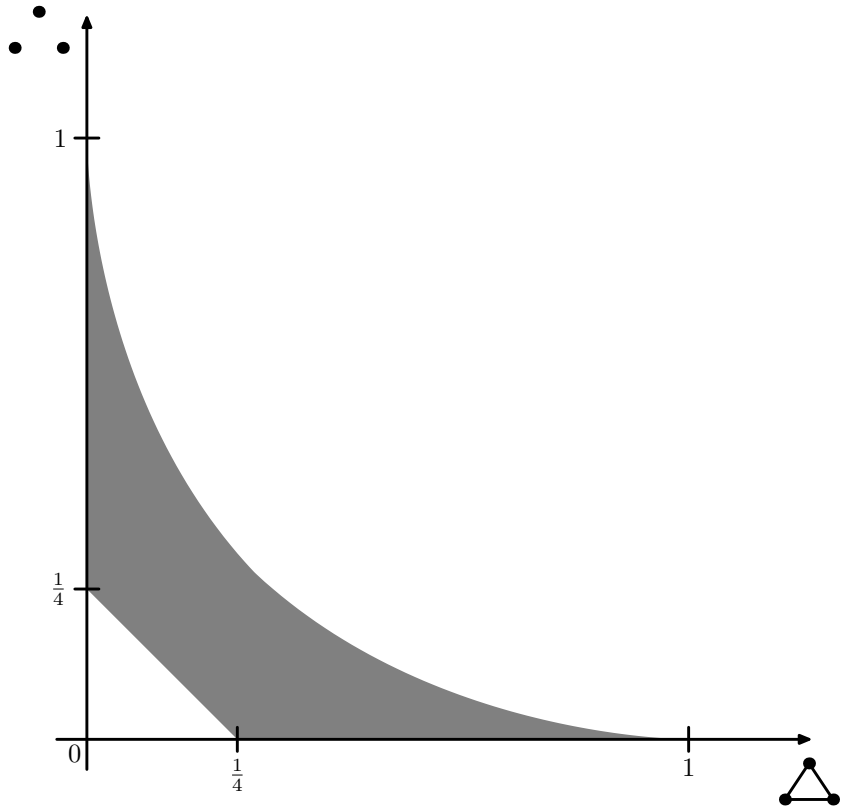}
\end{center}
\caption{Possible densities of triangles and co-triangle in graphs.}
\label{fig-tr-ct}
\end{figure}

In this paper, we determine the projections $S_{ij}$ for all the other pairs of $i$ and $j$.
By considering the complements of graphs,
it is easy to see that the projections $S_{xy}$ and $S_{(3-x)(3-y)}$ are the same.
Hence, it is enough to consider the following three projections only: $S_{12}$, $S_{13}$, and $S_{23}$,
which we consider separately in Sections~\ref{sec-tr-cc}--\ref{sec-tr-ch}.
While determining the projections $S_{12}$ and $S_{23}$ turned out to be relatively straightforward,
the projection $S_{13}$, which we consider in Section~\ref{sec-tr-cc}, was significantly more difficult to describe.
We would like to stress that although some of the proofs in Section~\ref{sec-tr-cc} use the flag algebra method,
all our proofs are computer-free and the method presents for us a very convenient way of formulating our arguments.
In the concluding Section~\ref{sec-concl},
we briefly discuss the possible structure of graphs with densities on the boundaries of the projections.

\section{Graph limits and flag algebras}
\label{sec-limits}

In this section, we briefly introduce the theory of graph limits and the flag algebra method.
The flag algebra method can be presented independently of the theory of graph limits
but since we will only apply the flag algebra method in the limit setting,
we find it more convenient to introduce it using some graph limit notation.
We also deal here only with limits of dense graphs, which we need in this paper, and
refer the reader to a recent monograph by Lov\'asz~\cite{bib-lovasz-book}
for a more detailed exposition of graph limits.

Recall that the {\em density} $d(H,G)$ of a $k$-vertex graph $H$ in a graph $G$ is the probability that
a randomly chosen subset of $k$ vertices of $G$ induces a subgraph isomorphic to $H$.
A sequence $(G_n)_{n\in\NN}$ of graphs is {\em convergent}
if the sequence of densities $d(H,G_n)$ converges for every graph $H$.
In what follows, we will always assume that if $(G_n)_{n\in\NN}$ is a convergent sequence,
then the number of vertices of the graphs in the sequence tends to infinity.

Convergent sequences of graphs can be represented by an analytic object called a graphon;
A {\em graphon} is a symmetric measurable function $W:[0,1]^2\to [0,1]$,
where symmetric stands for the property that $W(x,y)=W(y,x)$ for all $x,y\in [0,1]$.
We next define a density of a $k$-vertex graph $H$ in a graphon $W$ as follows.
A $k$-vertex {\em $W$-random graph} is the random graph obtained
by sampling $k$ points $x_1,\ldots,x_n$ independently and uniformly in the unit interval $[0,1]$ and
joining the $i$-th vertex and the $j$-th vertex of the graph by an edge with probability $W(x_i,x_j)$.
The {\em density} of a $k$-vertex graph $H$ in $W$ is the probability that
a $k$-vertex $W$-random graph is isomorphic to $H$; this density is denoted by $d(H,W)$. 
To facilitate reading we often use $d(H,W)$ where $H$ is a drawing of the graph.
We will sometimes refer to the elements of $[0,1]$ as to the vertices of $W$, and
we will say that the {\em degree} $d_W(x)$ of $x\in [0,1]$ is the integral
$$\int_{[0,1]} W(x,y)\;\mbox{d}y\,\mbox{,}$$
which is the expectation of the fraction of the vertices in a $W$-random graph adjacent to
a vertex associated with $x$.

A graphon $W$ is a {\em limit} of a convergent sequence $(G_n)_{n\in\NN}$ of graphs
if the density $d(H,W)$ is equal to the limit of the densities $d(H,G_n)$ for every graph $H$.
It is known that every convergent sequence of graphs has a limit~\cite{bib-lovasz06+}, and
this limit is unique up to certain measure preserving transformations~\cite{bib-borgs10+}.
In what follows, we use $|X|$ to denote the measure of a subset $X\subseteq [0,1]$.
If $W$ is a graphon and $A$ a non-null subset of $[0,1]$,
we define a graphon $W[A]$ {\em induced} by $A$ as follows:
fix any mapping $\varphi:[0,1]\to A$ such that $|X|=|\varphi^{-1}(X)|\cdot|A|$ for every measurable $X\subseteq A$ and
set $W[A](x,y)=W(\varphi(x),\varphi(y))$.

A {\em component} of a graphon $W$ is a non-null subset $A\subseteq [0,1]$ such that
$W$ is equal to zero almost everywhere on $A\times\overline{A}$ and
there is no subset $B$ of $A$ such that both $B$ and $A\setminus B$ are non-null and
$W$ is equal to zero almost everywhere on $B\times (A\setminus B)$.
It is not hard to show that for every graphon $W$,
there is a countable (possibly finite) collection $(C_j)_{j\in J}$ of disjoint subsets of $[0,1]$ such that
each $C_j$ is a component of $W$ and $W$ is zero almost everywhere outside $\cup_{j\in J}C_j^2$.
Note that the set $[0,1]\setminus\cup_{j\in J}C_j$ may be non-null.

The theory of graph limits can be used to present an alternative definition of the set $S$.
A point $(d_0,d_1,d_2,d_3)$ belongs to the set $S$ if and only if
there exists a graphon $W$ such that $d(H_i,W)=d_i$ for $i=0,1,2,3$.
Since the numbers $d_0$, $d_1$, $d_2$ and $d_3$ present the probabilities of sampling $3$-vertex $W$-random graphs,
the following becomes obvious.
\begin{proposition}
\label{prop-unit}
If $(d_0,d_1,d_2,d_3) \in S$, then \hbox{$d_0+d_1+d_2+d_3=1$}.
\end{proposition}
Since the edge density of a graph is the average edge density in its $3$-vertex subgraphs,
we also get the following.
\begin{proposition}
\label{prop-edge}
Let $(G_n)_{n\in\NN}$ be a convergent sequence of graphs, and
let $d_e$ be its limit edge density and $d_k$ the limit density of the graph $H_k$.
It holds that $d_e=\frac{d_1+2d_2+3d_3}{3}$.
\end{proposition}

We now introduce basic concepts related to the flag algebra method,
which was developed by Razborov~\cite{bib-razborov07}.
The method has become a popular tool in extremal combinatorics,
see, e.g., \cite{bib-flag1, bib-flag2, bib-flagrecent, bib-flag3, bib-flag4, bib-flag5, bib-flag6, bib-flag7, bib-flag8, bib-pikhurko16+, bib-flag9, bib-flag12, bib-flag11}, and
led to solving many long standing open problems in the area.
In our exposition, we focus on presenting the main concepts only, and
refer the reader for a more detailed exposition to,
e.g., the original paper of Razborov~\cite{bib-razborov07}.

Let $\FF$ be a set of finite formal linear combinations of graphs.
We represent elements of $\FF$ as formal linear combinations of drawings of the corresponding graphs,
e.g., $\flage-\frac{1}{2}\flagHzero$.
For a graphon, let $h_W:\FF\to\RR$ be the mapping that
$$h_W\left(\sum_{j\in J} \alpha_j G_j\right)=\sum_{j\in J}\alpha_j d(G_j,W)\;\mbox{.}$$
The mapping $h_W$ respects both the addition of elements of $\FF$ and the multiplication by a scalar.
Razborov~\cite{bib-razborov07} showed that it is possible to define a multiplication of the elements of $\FF$ that
the mapping also respects this operation, i.e., $h_W(x\times y)=h_W(x)h_W(y)$ for all $x,y\in\FF$.
To keep our notation simple, we will occasionally write $d(x,W)$ instead of $h_W(x)$.

Suppose that $x,y\in\FF$. 
We write that $x=y$ if $h_W(x)=h_W(y)$;
if a graphon $W$ is not specified, we mean that the equality holds for all graphons.
For example, Proposition~\ref{prop-edge} can be rewritten as
$\flage=\frac{1}{3}\flagHone+\frac{2}{3}\flagHtwo+\flagHthree$.
Likewise, we write $x\le y$ if $h_W(x)\le h_W(y)$.
In particular, $0\le y$ if $h_W(y)$ is non-negative.

We now extend the just introduced concepts to rooted graphs.
Let $H$ be a graph that has $k$ vertices and these are labelled with integers $1,\ldots,k$.
An {\em $H$-rooted graph} is a graph with $k$ of its vertices labelled with $1,\ldots,k$ in such a way that
the labelled vertices induce a copy of $H$ in a way that preserves the labels.
Let $\FF^H$ be the set of formal linear combinations of $H$-rooted graphs.
In the analogy to $\FF$, we depict elements of $\FF^H$ as linear combinations of drawing of $H$-rooted graphs
where the roots are depicted with empty circles.
If $H$ has two or more vertices, all the elements of the sum have the copy of $H$ depicted in the same way.
For example, we will write $\flageroot-\frac{1}{2}\flagfroot$ or $\flagcherrytwo+\flagcherryone$.

Let $z_1,\ldots,z_k\in [0,1]$.
We define a mapping $h_{W,z_1,\ldots,z_k}:\FF^H\to\RR$ as follows.
If $G$ is an $H$-rooted graph with $n$ unlabeled vertices,
then $h_{W,z_1,\ldots,z_k}(G)$ is the probability that a $W$-random graph is $G$
conditioned on that the first $k$ of the vertices $x_1,\ldots,x_{k+n}\in [0,1]$ being $z_1,\ldots,z_k$ and
on that they induce a copy of $H$ preserving the labels, i.e., the vertex $x_i=z_i$ is labelled with $i$.
Note that the mapping $h_{W,z_1,\ldots,z_k}$ might not be defined for certain $k$-tuples $z_1,\ldots,z_k$.
We extend the mapping $h_{W,z_1,\ldots,z_k}$ by linearity to the whole set $\FF^H$.
Again, one can define the multiplication on $\FF^H$ in a way that the mapping respects the multiplication.

In the analogy with our earlier notation,
we write $x=y$, $x\le y$ and $0\le y$ for $x,y\in\FF^H$
if the (in)equality holds for almost every choice of the roots $z_1,\ldots,z_k$ for which 
the mapping $h_{W,z_1,\ldots,z_k}$ is well-defined.
It is possible to define a mapping $\unlabel{\cdot}{H}$ from $\FF^H$ to $\FF$ such that
the following holds for every $x\in\FF^H$ and every graphon $W$
$$\EE_{z_1,\ldots,z_k}h_{W,z_1,\ldots,z_k}(x)=h_W\left(\unlabel{x}{H}\right)\,\mbox{,}$$
where the expectation is taken with respect to the probability
with the density function proportional to the probability that $z_1,\ldots,z_k$ induce a copy of $H$.
If the graph $H$ is clear from the context, we write $\unlabel{x}{}$ instead of $\unlabel{x}{H}$.
Since the square of any number of non-negative, the following proposition easily follows.

\begin{proposition}
\label{prop-square}
Let $H$ be a labeled graph.
For every $x\in\FF^H$, it holds that $0\le\unlabel{x^2}{}$.
\end{proposition}

We finish this section with a simple example.
Suppose that $H$ is the single-vertex graph with its only vertex labelled with $1$.
If $W$ is a graphon and $z\in [0,1]$, then
$$h_{W,z}(\flageroot)=\int_{0,1} W(z,x)\mbox{d}x\;\mbox{.}$$
For example, if $W$ is the graphon that is equal to one on $[0,1/3]^2\cup (1/3,1]^2$ and
to zero elsewhere, then $h_{W,z}(\flageroot)=1/3$ if $z\in [0,1/3)$, and $h_{W,z}(\flageroot)=2/3$, otherwise.
Note that $\EE_{z}h_{W,z}\left(\flageroot\right)=h_W\left(\flage\right)=5/9$;
we remark that it holds that $\unlabel{\flageroot}{}=\flage$.

\section{Triangle density}
\label{sec-tr}

In this section, we briefly recall some results on the minimum triangle density
in large graphs with bounds on their minimum edge density.
Perhaps, the oldest bound of this type is the bound of Goodman~\cite{bib-goodman59},
which can be written using the flag algebra language as 
$$\flage(2\flage-1)\le\flagHthree.$$
It has been a long-standing open problem to determine the optimum function $g:\RR\to\RR$ such that
every graphon with edge density being $d_e$ has the triangle density at least $g(d_e)$.
The value of the function was known for $d_e\le 2/3$ due to work of Fisher~\cite{bib-fisher89}
until Razborov~\cite{bib-flag11} has solved this problem using his flag algebra method.

To state Razborov's result, we define the function $g_R:[0,1]\to [0,1]$ as follows.
Set $g_R(d_e)=0$ if $d_e\in [0,1/2)$ and $g_R(1)=1$.
If $d_e\in [1/2,1)$, let $k$ be the smallest integer such that $d_e\le 1-1/k$.
Let $z$ be the unique integer in the interval $(0,1/k]$ such that
$$d_e = 1-(k-1)\cdot\left(\frac{1-z}{k-1}\right)^2-z^2=\frac{(1-z)(kz+k-2)}{k-1}\;\mbox{,}$$
and define 
\begin{eqnarray*}
g_R(d_e) &=& 6\binom{k-1}{3}\left(\frac{1-z}{k-1}\right)^3+6\binom{k-1}{2}z\left(\frac{1-z}{k-1}\right)^2\\
		 &=& \frac{(1-z)^2(k-2)(2zk+k-3)}{(k-1)^2}\;\mbox{.}
\end{eqnarray*}
Note that $g_R(d_e)$ is the asymptotic triangle density in a complete $k$-partite graph such that
one of its parts contain the fraction $z$ of its vertices and the remaining $k-1$ parts have the same size.
Note that $g_R(d_e)=0$ if $d_e\in [0,1/2]$ and
\begin{equation}
g_R(d_e)=\frac{(1-\sqrt{4-6d_e})(2+\sqrt{4-6d_e})^2}{18}
\label{eq-g3}
\end{equation}
for $d_e\in [1/2,2/3]$.
Razborov's result is equivalent to the following.

\begin{theorem}[Razborov~\cite{bib-flag11}]
\label{thm-e-tr}
If $W$ is a graphon with edge density $d_e$,
then the triangle density of $W$ is at least $g_R(d_e)$.
In particular, if $(d_0,d_1,d_2,d_3)\in S$,
then $d_3\ge g_R((d_1+2d_2+3d_3)/3)$.
\end{theorem}

The set of feasible edge and triangle densities is depicted in Figure~\ref{fig-e-tr}.

\begin{figure}[ht]
\begin{center}
\epsfbox{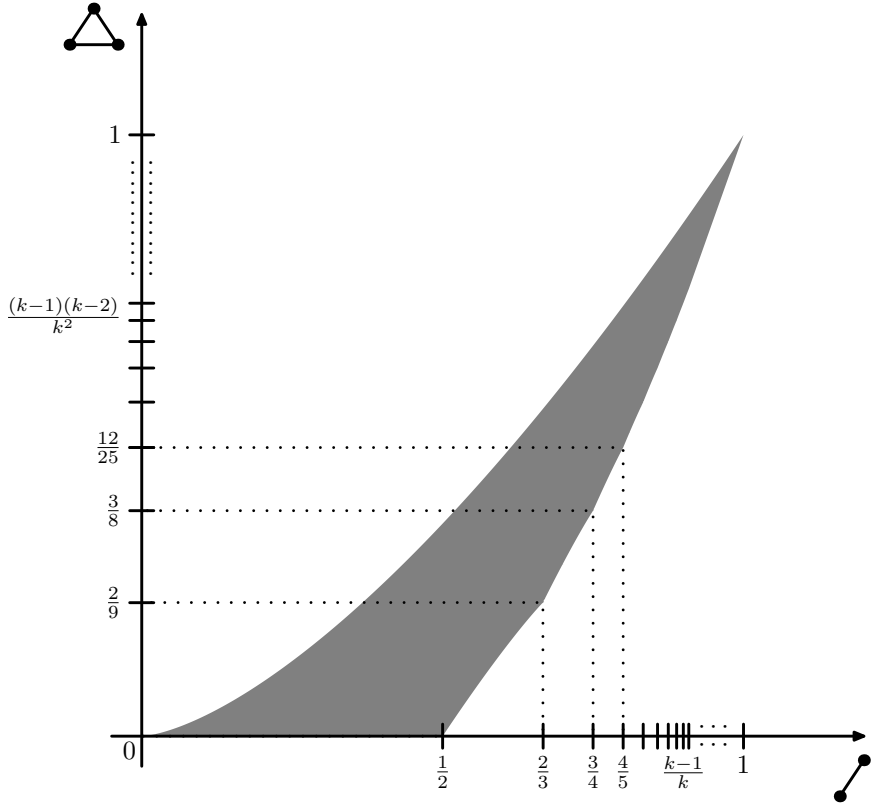}
\end{center}
\caption{Possible densities of edges and triangles in graphs.}
\label{fig-e-tr}
\end{figure}

Pikhurko and Razborov~\cite{bib-pikhurko16+} described structure of large graphs
with edge density $d_e$ and triangle density $g_R(d_e)$.
We state their result in the language of graphons.
Let $d_e$ be in $[1/2,1)$.
Let $k$ be the smallest integer such that \hbox{$d_e\le 1-1/k$}, and let $z$ be as in the definition of the function $g_R$.
Further, denote \hbox{$C_i=\left[\frac{(i-1)(1-z)}{k-1},\frac{i(1-z)}{k-1}\right)$} for $i=1,\ldots,k-2$ and
$C_{k-1}=\left[\frac{(k-2)(1-z)}{k-1},1\right]$.
Define a graphon $W$ such that $W(x,y)=0$ if $(x,y)\in C_i^2$ for $i=1,\ldots,k-2$,
$W(x,y)=1$ if $(x,y)\in C_i\times C_j$ for $i\not=j$, and
$W[C_{k-1}]$ is a graphon with the zero triangle density, and
the edge density equal to
$$\frac{2z(k-1)(1-z)}{(kz-2z+1)^2}\,\mbox{,}$$
i.e., the edge density of the complete bipartite graph with parts containing the fractions of
$\frac{(1-z)/(k-1)}{z+(1-z)/(k-1)}$ and $\frac{z}{z+(1-z)/(k-1)}$ of vertices.
For any choice of $W[C_{k-1}]$,
the graphon $W$ has the edge density equal to $d_e$ and
the triangle density equal to $g_R(d_e)$.
Pikhurko and Razborov~\cite{bib-pikhurko16+} showed that
if $(G_n)_{n\in\NN}$ is a convergent sequence of graphs such that
its limit edge density is $d_e\in [1/2,1)$ and its triangle density is $g_R(d_e)$,
then its limit is one of the graphons $W$ defined above.

A stronger lower bound on the number of triangles can be shown assuming that
every vertex is adjacent to at least the fraction $d\in [0,1]$ of all vertices.
Lo~\cite{bib-lo12} proved tight structural results if $d\in [0,3/4]$.
We state his result in the complementary form that we apply later in our considerations.

\begin{theorem}[Lo~\cite{bib-lo12}]
\label{thm-lo}
Let $d\in [1/4,1]$, and let $W$ be a graphon that minimizes $d(\flagHzero,W)$
subject to that $d_W(x)\le d$ for almost every $x\in [0,1]$.
It holds that $d_W(x)=d$ for almost every $x\in [0,1]$ and
\begin{itemize}
\item if $d\in (1/2,1]$, then graphon $W$ has a single component of measure one and $d(\flagHzero,W)=0$,
\item if $d\in (1/3,1/2]$, then $W$ has two components $C_1$ and $C_2$ of measures $|C_1|=d$ and $|C_2|=1-d$,
      $W$ is equal to one almost everywhere on $C_1^2$ and $d(\flagHzero,W[C_2])=0$,
\item if $d\in (1/4,1/3]$, then $W$ has three components $C_1$, $C_2$ and $C_3$ of measures $|C_1|=|C_2|=d$ and $|C_3|=1-2d$,
      $W$ is equal to one almost everywhere on $C_1^2\cup C_2^2$ and
      $d(\flagHzero,W[C_3])=0$, and
\item if $d=1/4$, then graphon $W$ has four components $C_1,\ldots,C_4$, each of the components has measure $1/4$, and
      $W$ is equal to one almost everywhere on \hbox{$C_1^2\cup C_2^2\cup C_3^2\cup C_4^2$}.
\end{itemize}
\end{theorem}

\section{Triangle vs.~co-cherry projection}
\label{sec-tr-cc}

In this section, we determine the projection $S_{13}$.
We start with defining several auxiliary functions.
The first two functions represent the asymptotic densities of co-cherries and triangles in graphs that have the following structure.
Let $\sigma\in (1/4,1/3]$. The graph have three components:
two of the components are complete graphs on the fraction $\sigma$ of all vertices, and
the remaining component consists of two cliques $C_1$ and $C_2$,
each formed by the fraction of $(1-2\sigma)/2$ of all vertices, such that
each vertex of $C_i$ is adjacent to the fraction of
\begin{equation}
\delta_A(\sigma)=\frac{4\sigma-1}{(1-2\sigma)\sqrt{5-12\sigma}} \label{eq-dA}
\end{equation}
of the vertices of $C_{3-i}$,
i.e., each vertex of this component is adjacent to the fraction $(1+\delta_A(\sigma))\frac{1-2\sigma}{2}$ of all vertices of the graph.
An example of such a graph can be found in Figure~\ref{fig-graph-hA}.

\begin{figure}[ht]
\begin{center}
\epsfbox{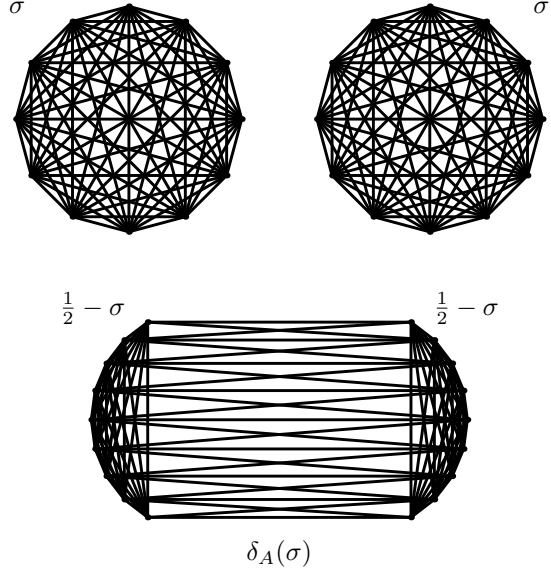}
\end{center}
\caption{A graph constructed in Section~\ref{sec-tr-cc}, which has co-cherry density $h_{A,1}(\sigma)$ and triangle density $h_{A,3}(\sigma)$, where $\sigma\in (1/4,1/3]$.}
\label{fig-graph-hA}
\end{figure}

The asymptotic co-cherry and triangle densities $h_{A,1}$ and $h_{A,3}$ in graphs having this structure are the following:
\begin{eqnarray*}
h_{A,1}(\sigma) 
                & = & \frac{9-48\sigma+114\sigma^2-120\sigma^3+3(1-2\sigma)(4\sigma-1)^2\sqrt{5-12\sigma}}{10-24\sigma} \\
h_{A,3}(\sigma) & = & \frac{2-18\sigma+57\sigma^2-60\sigma^3}{5-12\sigma}
\end{eqnarray*}
Note that $h_{A,1}(1/4)=9/16$ and $h_{A,3}(1/4)=1/16$;
these are the asymptotic co-cherry and triangle densities in a graph consisting of four equal size complete graphs,
which is the ``limit structure'' of the above graphs as $\sigma$ tends to $1/4$.
Also note that both $h_{A,1}$ and $h_{A,3}$ are increasing functions of $\sigma$ in the interval $[1/4,1/3]$ and
their co-domains are $[9/16,2/3]$ and $[1/16,1/9]$, respectively.

The next two functions $h_{B,1}$ and $h_{B,3}$
are asymptotic co-cherry and triangle densities in a graph consisting of three complete graphs,
two formed by $\sigma$ fraction of all vertices each, and
the remaining one formed by $1-2\sigma$ of the fraction of all vertices,
where $\sigma\in [1/3,1/2)$. 


The functions $h_{B,1}$ and $h_{B,3}$ are defined as follows:
\begin{eqnarray*}
h_{B,1}(\sigma) & = & 6\sigma-18\sigma^2+18\sigma^3 \\
h_{B,3}(\sigma) & = & 1-6\sigma+12\sigma^2-6\sigma^3 
\end{eqnarray*}
Note that $h_{A,1}(1/3)=h_{B,1}(1/3)=2/3$ and $h_{A,3}(1/3)=h_{B,3}(1/3)=1/9$.
In addition, $h_{B,1}(1/2)=3/4$ and $h_{B,3}=1/4$,
which are the asymptotic co-cherry and triangle densities in a graph formed by two equal size cliques.
Also note that both $h_{B,1}$ and $h_{B,3}$ are increasing functions of $\sigma$ in the interval $[1/3,1/2]$
with co-domains $[2/3,3/4]$ and $[1/9,1/4]$, respectively.

\begin{figure}[ht]
\begin{center}
\epsfbox{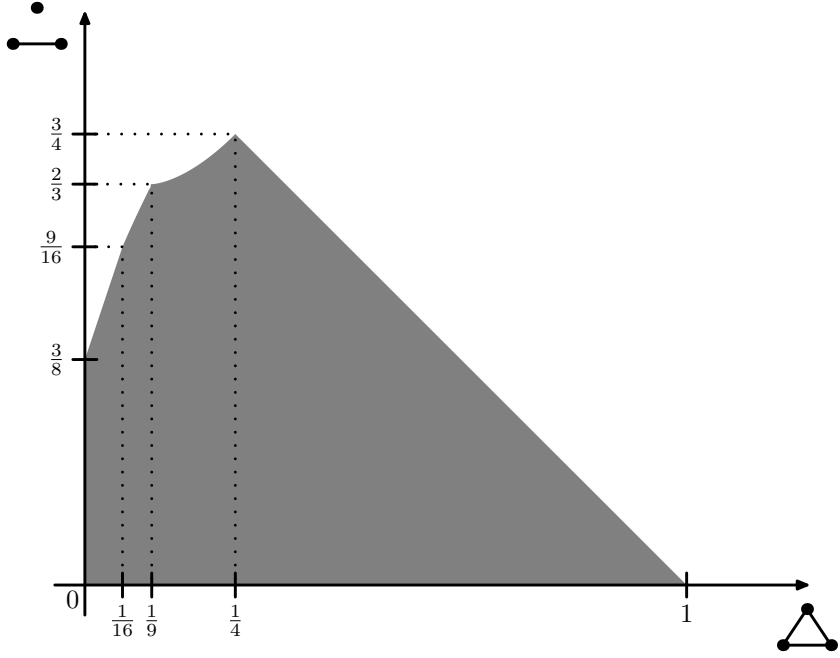}
\end{center}
\caption{Possible densities of triangles and co-cherries in graphs.}
\label{fig-tr-cc}
\end{figure}

We are now ready to define a function $g_t:[0,1]\to [0,1]$,
which will determine the most complex part of the boundary of the projection $S_{13}$.
The function $g_t$ is defined as follows:
\begin{equation*}
g_t(x)=\left\{\begin{array}{cl}
       3x+\frac{3}{8} & \mbox{if $x\in [0,1/16]$,}\\
       h_{A,1}(h_{A,3}^{-1}(x)) & \mbox{if $x\in (1/16,1/9)$,}\\
       h_{B,1}(h_{B,3}^{-1}(x)) & \mbox{if $x\in [1/9,1/4)$, and}\\
       1-x & \mbox{otherwise.}
       \end{array}
       \right.
\end{equation*}
We prove in Theorem~\ref{thm-tr-cc} that the projection $S_{13}$
is equal to the set of the points $(d_1,d_3)$ such that $0\le d_1\le g_t(d_3)$ and $d_3\in [0,1]$.
This set is visualized in Figure~\ref{fig-tr-cc}.

The proof of Theorem~\ref{thm-tr-cc} is split into several steps.
We establish three lemmas that provide different upper bounds on $d_1$ in terms of $d_3$;
each of the upper bounds is tight for a different range of the values $d_3$.
We start with proving the simplest of the lemmas,
which yields the tight upper bound on the initial linear segment
of the upper bound on $d_1$ depicted in Figure~\ref{fig-tr-cc}.



\begin{lemma}
\label{lm-curve-A}
Every point $(d_1,d_3)$ contained in $S_{13}$ satisfies $d_1\le 3d_3+3/8$.
\end{lemma}

\begin{proof}
The statement of the lemma is equivalent to showing the following inequality
in the flag algebra language.
$$\flagHone\le3\flagHthree+\frac{3}{8}$$
By Proposition~\ref{prop-square}, we get the following:
\begin{equation}
0\le\unlabel{\left(3\flageroot-\flagfroot\right)^2}{}=9\flagHthree+\flagHtwo-\frac{5}{3}\flagHone+\flagHzero
\label{eq-A-1}
\end{equation}
Since $1=\flagHzero+\flagHone+\flagHtwo+\flagHthree$, it holds that
$$0\le 8\flagHthree-\frac{8}{3}\flagHone+1\;\mbox{,}$$
which yields the desired inequality $\flagHone\le3\flagHthree+3/8$.
\end{proof}

In the following two subsections, we show that the function $g_t(d_3)$ provides
the tight upper bound on $d_1$ for $d_3 \in (1/16, 1/9)$ and $d_3 \in [1/9,1/4)$,
respectively.

\subsection{The concave regime of $g_t$}

In this subsection, we bound $d_1$ in terms of $d_3$ for $d_3 \in (1/16,1/9)$.
Before doing so, we need to study the structure of graphons maximizing a certain linear combination of $3$-vertex graphs.
We start with a lemma saying that almost any two adjacent vertices in such graphons have the same degree.

\begin{lemma}
\label{lm-regular}
Let $\alpha\in [1,3]$, and let $W$ be a graphon maximizing $d(\flagHone-\alpha\flagHthree,W)$.
It holds for almost every pair $(x,y)\in [0,1]^2$ with $W(x,y)>0$ that $d_W(x)=d_W(y)$.
\end{lemma}

\begin{proof}
The statement of the lemma is equivalent to establishing the following equality.
\begin{equation}
\unlabel{\left((\flagErootone+\flagErootboth)-(\flagEroottwo+\flagErootboth)\right)^2}{}=
\unlabel{\left(\flagErootone-\flagEroottwo\right)^2}{}=0
\label{eq-reg-1}
\end{equation}
Indeed, if $x$ and $y$ are the two root vertices in (\ref{eq-reg-1}),
then the expression on the left side of (\ref{eq-reg-1}) is equal to $\unlabel{\left(d_W(x)-d_W(y)\right)^2}{}$.
This implies that the left side of (\ref{eq-reg-1}) is equal to
$$\int_{(x,y)\in [0,1]^2, W(x,y)>0} \left(d_W(x)-d_W(y)\right)^2\;\mbox{d}x\,\mbox{d}y\,\mbox{;}$$
this integral is zero if and only if the assertion of the lemma holds.
Hence, we need to prove (\ref{eq-reg-1}).

We now derive several inequalities using the differential method described in \cite[Subsection~4.3]{bib-razborov07} that
must be satisfied by any graphon $W$ maximizing $d(\flagHone-\alpha\flagHthree,W)$.
The first inequality corresponds to deleting an edge between two distinguished vertices, and
it directly follows from~\cite[Theorem~4.5]{bib-razborov07}.
The inequality (\ref{eq-reg-2}) holds for almost every choice of an edge.
\begin{equation}
0\ge\partial_{\flagEroot}(\flagHone-\alpha\flagHthree)=3(\flagErootone+\flagEroottwo-\flagErootzero+\alpha\flagErootboth)
\label{eq-reg-2}
\end{equation}
Using (\ref{eq-reg-2}), we derive that
\begin{eqnarray}
0 & \le & 6\unlabel{(\flagErootzero-\flagErootone-\flagEroottwo-\alpha\flagErootboth)\times(\flagErootone+\flagEroottwo)}{} = \nonumber\\
  & & \flagfourtwo-3\flagfourthree-(1+\alpha)\flagfourseven-4\flagfoureight-2\alpha\flagfournine\,\mbox{,}
\label{eq-reg-3}
\end{eqnarray}
which also follows from~\cite[Corollary~4.6]{bib-razborov07}.
We next consider the operation of adding an edge and we obtain following the lines of reasoning for deleting an edge that
the following holds for almost every choice of a non-edge.
\begin{equation}
0\ge\partial_{\flagFroot}(\flagHone-\alpha\flagHthree)=3(\flagFrootzero-\flagFrootone-\flagFroottwo-\alpha\flagFrootboth)
\label{eq-reg-4}
\end{equation}
Using (\ref{eq-reg-4}), we get that it holds that
\begin{eqnarray}
0 & \le & 12\unlabel{(\flagFrootone+\flagFroottwo-\flagFrootzero+\alpha\flagFrootboth)\times\flagFrootboth}{} = \nonumber\\
  & & -\flagfourtwo-3\flagfourthree+2\flagfoursix+2\flagfourseven+4\alpha\flagfoureight+2\alpha\flagfournine\;\mbox{.}
\label{eq-reg-5}
\end{eqnarray}
The final operation that we consider is the following operation:
consider a cherry labelled in such a way that the leaves are the first and the third vertices.
The operation that we consider is the operation of removing the edge between the second and third vertices, and
adding the edge between the first and the third vertices.
Following the way the inequalities (\ref{eq-reg-2}) and (\ref{eq-reg-4}) were derived,
we obtain that the following holds for almost every choice of a cherry.
$$0\ge \flagcherrytwo-\flagcherryone+\alpha\flagcherrytwothree-\alpha\flagcherryonethree$$
Hence, it holds that
\begin{equation}
0 \le 12\unlabel{\flagcherryone-\flagcherrytwo+\alpha\flagcherryonethree-\alpha\flagcherrytwothree}{}=
      -3\flagfourthree+\flagfoursix-\alpha\flagfourseven+4\alpha\flagfoureight\;\mbox{.}
\label{eq-reg-7}
\end{equation}
We get the following inequality by summing
the inequality (\ref{eq-reg-3}) multiplied by $\frac{\alpha-1}{\alpha+1}$,
(\ref{eq-reg-5}) multiplied by $\frac{\alpha-1}{\alpha+1}$, and
(\ref{eq-reg-7}) multiplied by $\frac{3-\alpha}{\alpha+1}$.
\begin{equation}
0 \le -3\flagfourthree+\flagfoursix-\flagfourseven+4\flagfoureight
\label{eq-reg-8}
\end{equation}
Since it holds that
$$\unlabel{\left(\flagErootone-\flagEroottwo\right)^2}{}=\frac{1}{2}\flagfourthree-\frac{1}{6}\flagfoursix
                                                        +\frac{1}{6}\flagfourseven-\frac{2}{3}\flagfoureight\,\mbox{,}$$
which is a multiple of $-1/6$ of the right side of (\ref{eq-reg-8}),
the left side of (\ref{eq-reg-1}) cannot be positive.
Since the left side of (\ref{eq-reg-1}) is non-negative by Proposition~\ref{prop-square},
the equality (\ref{eq-reg-1}) now follows, which finishes the proof of the lemma.
\end{proof}

The next lemma concerns degrees of non-adjacent vertices.

\begin{lemma}
\label{lm-min-degree}
Let $\alpha\in [1,3)$, and let $W$ be a graphon maximizing $d(\flagHone-\alpha\flagHthree,W)$.
Then for almost every pair $(x,y)\in [0,1]^2$ with $W(x,y)<1$ the inequality \hbox{$d_W(x)+d_W(y)\ge 1/2$} holds.
\end{lemma}

\begin{proof}
Following the reasoning in the proof of Lemma~\ref{lm-regular},
we get that almost every choice of a non-edge,
i.e., almost every pair $(x,y)\in [0,1]^2$ with $W(x,y)<1$, satisfies (\ref{eq-reg-4}).
Hence, it holds that
$$0 \le \alpha\flagFrootboth+\flagFrootone+\flagFroottwo-\flagFrootzero=
        (\alpha+1)\flagFrootboth+2\flagFrootone+2\flagFroottwo-1 \le 
	4\flagFrootboth+2\flagFrootone+2\flagFroottwo-1 $$
for almost every choice of non-edge.
Since $\flagFrootone+\flagFroottwo+2\flagFrootboth$ is equal to $d_W(x)+d_W(y)$,
the lemma now follows.
\end{proof}

To get the desired bound on the co-cherry density,
we analyze an optimization problem,
which involves the derivative of the function $g_t$ in the interval $(1/16,1/9)$.
Before stating the lemma, it is useful to compute this derivative.
We start by investigating the derivatives of the functions $h_{A,1}$ and $h_{A,3}$:
\begin{eqnarray*}
h'_{A,1}(\sigma) & = & \frac{6(4\sigma-1)(60\sigma^2-51\sigma+11)(1+\sqrt{5-12\sigma})}{(5-12\sigma)^2}\\
h'_{A,3}(\sigma) & = & \frac{6(4\sigma-1)(60\sigma^2-51\sigma+11)}{(5-12\sigma)^2}
\end{eqnarray*}
It now follows that the derivative of the function $g_t$ at a point $x\in (1/16,1/9)$ is the following:
$$g'_t(x)=1+\sqrt{5-12h_{A,3}^{-1}(x)}\;\mbox{.}$$
In particular, $g'_t$ is decreasing in the interval $(1/16,1/9)$ and $2<g'_t(x)<1+\sqrt{2}$.

\begin{lemma}
\label{lm-curve-B}
Let $x\in (1/16,1/9)$.
The following inequality holds for every point $(d_1,d_3)\in S_{13}$:
$$d_1-g'_t(x)d_3\le g_t(x)-g'_t(x)x\;\mbox{.}$$
In particular, $d_1\le g_t(d_3)$ for every point $(d_1,d_3)\in S_{13}$ with $d_3\in (1/16,1/9)$.
\end{lemma}

\begin{proof}
Fix $x\in (1/16,1/9)$ for the proof.
We need to show that if $W$ is a graphon maximizing $d(\flagHone-g'_t(x)\flagHthree,W)$,
then $d(\flagHone-g'_t(x)\flagHthree,W)\le g_t(x)-g'_t(x)x$.
Set $\alpha=g'_t(x)$ and
fix a graphon $W_0$ maximizing $d(\flagHone-\alpha\flagHthree,W)$.
Further, let $(C_j)_{j\in J}$ be the set of components of graphon $W_0$ with positive measure.

Let $D=[0,1]\setminus \cup_{j\in J}C_j$.
Observe that $W_0$ is zero almost everywhere on $D\times [0,1]$.
In particular, $d_{W_0}(x)=0$ for almost every $x\in D$.
Hence, Lemma~\ref{lm-min-degree} implies that $D$ is null.
Consequently, we can assume that $\cup_{j\in J}C_j$ is equal to $[0,1]$
since the set can be added to one of the components in the collection $(C_j)_{j\in J}$
without violating the constraints.

Suppose that $|J|\not=1$.
Since $d_{W_0}(x)\le |C_j|$ for every $j\in J$ and almost every $x\in C_j$, 
Lemma~\ref{lm-min-degree} implies that $|C_j|+|C_{j'}| \ge 1/2$ for every distinct $j,j'\in J$.
This yields that the number of components is at most four;
otherwise, there would be two components with the sum of their measures less than $1/2$.

If $W_0$ has exactly four components, then Lemma~\ref{lm-min-degree} gives that they are of equal size and each of them is a clique. For such $W_0$ we have $d_3 = 1/16$ and $d_1 = g_t(1/16)$, then since $g_t$ is continuous in $[1/16, 1/9]$ and $g_t'$ is decreasing in $(1/16, 1/9)$, we have $g_t(x) \ge d_1 + g_t'(x)(x-d_3)$. So $W_0$ fulfills the
desired inequality.

Hence, we can assume that $J=\{1,\ldots,|J|\}$ and $|J|$ is one, two or three.

Lemma~\ref{lm-regular} implies that
there exist $\delta_j$, $j\in J$, such that $d_{W_0}(x)=\delta_j|C_j|$ for almost every $x\in C_j$.
Otherwise, there would exist a real $\delta'$, a partition of $C_j$ to two non-null sets $C'$ and $C''$ such that
$d_{W_0}(x)\le \delta'$ for almost every $x\in C'$ and $d_{W_0}(x)>\delta'$ for almost every $x\in C''$.
Since ${W_0}$ is not zero almost everywhere on $C'\times C''$,
Lemma~\ref{lm-regular} would yield a contradiction.

Let $t_j=d(\flagHthree,W_0[C_j])$ for $j\in J$.
Observe that
$$\flagHtwo=3\unlabel{\flageroot^2}{}-3\flagHthree\mbox{,}\;
  \flagHone=3\flage-2\flagHtwo-3\flagHthree
  \mbox{ and }
  \flagHzero=1-\flagHthree-\flagHtwo-\flagHone
  \,\mbox{.}$$
Since $d_{W_0}(x)=\delta_j|C_j|$ for almost every $x\in C_j$, it follows that
\begin{eqnarray}
d(\flagHthree,W_0[C_j]) & = & t_j \nonumber\\
d(\flagHtwo,W_0[C_j]) & = & 3\delta_j^2-3t_j \nonumber\\
d(\flagHone,W_0[C_j]) & = & 3\delta_j-6\delta_j^2+3t_j \nonumber\\
d(\flagHzero,W_0[C_j]) & = & 1-3\delta_j+3\delta_j^2-t_j \label{eq-B-4}
\end{eqnarray}
In addition, note that $d(\flage,W_0[C_j]) = \delta_j$.
It now follows that
\begin{eqnarray}
d(\flagHthree,W_0) & = & \sum_{j\in J}t_j|C_j|^3 \label{eq-B-1}\\
d(\flagHone,W_0) & = & \sum_{j\in J} \left(3\delta_j-6\delta_j^2+3t_j\right)|C_j|^3+3\delta_j|C_j|^2\left(1-|C_j|\right)\label{eq-B-2}
\end{eqnarray}
The equalities (\ref{eq-B-1}) and (\ref{eq-B-2}) yield that
\begin{equation}
d(\flagHone-\alpha\flagHthree,W_0)=K+\sum_{j\in J} (3-\alpha)d(\flagHthree,W_0[C_j])|C_j|^3\;\mbox{,}\label{eq-B-3}
\end{equation}
where $K$ depends on $\delta_j$'s and $|C_j|$'s only. 

Fix $j\in J$ for this and the next paragraph.
We now show that the graphon $W_0[C_j]$ minimizes $d(\flagHzero,W)$ among all graphons $W$ such that $d_W(x)=\delta_j$ for almost every $x\in [0,1]$.
Suppose that there exists a graphon $W$ such that $d(\flagHzero,W)<d(\flagHzero,W_0[C_j])$ and
$d_W(x)=\delta_j$ for almost every $x\in [0,1]$, and
consider the graphon $W'$ that is equal to $W_0$ everywhere outside $C_j^2$ and that satisfies $W'[C_j]=W$.
By (\ref{eq-B-4}), it holds that $d(\flagHthree,W_0[C_j])<d(\flagHthree,W'[C_j])$,
and so $d(\flagHone-\alpha\flagHthree,W_0)<d(\flagHone-\alpha\flagHthree,W')$ by (\ref{eq-B-3})
since $\alpha\in (2,1+\sqrt{2})$.
However, this contradicts the choice of $W_0$.

We claim that $\delta_j>1/2$ and $d(\flagHzero,W_0[C_j])=0$.
If $W_0[C_j]$ is equal to one almost everywhere on $C_j^2$, then $\delta_j=1$ and the claim follows.
Otherwise, we can apply Lemmas~\ref{lm-regular} and \ref{lm-min-degree} to two vertices of $C_j$ and conclude that $2\delta_j|C_j| \ge 1/2$,
which yields that $\delta_j \ge 1/4$.
Since every graphon $W$ that minimizes $d(\flagHzero,W)$ among all graphons $W$ such that $d_W(x)\le \delta_j$
satisfies that $d_W(x)=\delta_j$ for almost every $x\in [0,1]$ by Theorem~\ref{thm-lo},
it follows that $W_0[C_j]$ is one of the graphons listed in the statement of Theorem~\ref{thm-lo}.
Since the graphon $W_0[C_j]$ has a single component, it follows that $\delta_j>1/2$ and  $d(\flagHzero,W_0[C_j])=0$.

Using (\ref{eq-B-4}), (\ref{eq-B-1}) and (\ref{eq-B-2}), we conclude that the following equalities hold.
\begin{eqnarray*}
d(\flagHthree,W_0) & = & \sum_{j\in J} \left(1-3\delta_j+3\delta_j^2\right)|C_j|^3 \label{eq-B-5} \\
d(\flagHone,W_0) & = & \sum_{j\in J} \left(3-9\delta_j+3\delta_j^2\right)|C_j|^3+3\delta_j|C_j|^2 \label{eq-B-6}
\end{eqnarray*}
We claim that $d(\flagHone-\alpha\flagHthree,W_0)$ is equal to the maximum value of the function
\begin{equation}
F(x_1,x_2,x_3,y_1,y_2,y_3)=\sum_{j=1}^3 x_j^3\left(3-\alpha-3(3-\alpha)y_j-3(\alpha-1)y_j^2\right)+3x_j^2y_j
\label{eq-B-7}
\end{equation}
subject to that
\begin{eqnarray}
x_1,x_2,x_3 & \ge & 0\,\mbox{,} \label{eq-B-8}\\
x_1+x_2+x_3 & = & 1\,\mbox{, and} \label{eq-B-9}\\
y_1,y_2,y_3 & \in & (1/2,1]\,\mbox{.} \label{eq-B-10}
\end{eqnarray}
Indeed, setting $x_j=|C_j|$, $y_j=\delta_j$ for $j\in J$, and 
setting $x_j=0$, $y_j=1$ for $j>|J|$ results in a feasible solution, and
the value of (\ref{eq-B-7}) is equal to $d(\flagHone-\alpha\flagHthree,W_0)$.
On the other hand, suppose that $x_1,x_2,x_3$ and $y_1,y_2,y_3$ is a feasible solution.
Let $C_1,C_2,C_3$ be arbitrary disjoint subsets of $[0,1]$ of measures $x_1,x_2,x_3$, respectively, and
let $W$ be the graphon such that for every $x_j>0$, $W[C_j]$ is a graphon such that
$d_{W[C_j]}(x)=y_j$ for almost every $x\in [0,1]$ and $d(\flagHzero,W[C_j])=0$;
for example, $W[C_j]$ can be chosen as the graphon equal to $1$ on $[0,1/2)^2\cup [1/2,1]^2$ and
equal to $2y_j-1$ elsewhere.
Since $d(\flagHone-\alpha\flagHthree,W)$ is equal to (\ref{eq-B-7}),
the choice of $W_0$ implies that the value of (\ref{eq-B-7})
does not exceed $d(\flagHone-\alpha\flagHthree,W_0)$.

The maximum value of (\ref{eq-B-7}) subject to (\ref{eq-B-8}), (\ref{eq-B-9}) and (\ref{eq-B-10})
is determined in Proposition~\ref{prop-max}, and
this value is indeed equal to $g_t(x)-\alpha\cdot x$,
in particular, it is equal to
\begin{eqnarray*}
h_{A,1}(\sigma)-\alpha\cdot h_{A,3}(\sigma) & = & \frac{24\sigma^3-18\sigma^2+6\sigma+\sqrt{5-12\sigma}-1}{2\sqrt{5-12\sigma}} \\
 & = & \frac{-\alpha^6+6\alpha^5-9\alpha^4-4\alpha^3+96\alpha-80}{144(\alpha-1)}\;\mbox{,}
\end{eqnarray*} 
where $\sigma=\frac{5-(\alpha-1)^2}{12}$.
\end{proof}

To complete the proof of Lemma~\ref{lm-curve-B},
we need to determine the maximum value of (\ref{eq-B-7}).

\begin{proposition}
\label{prop-max}
Let $\alpha\in(2,1+\sqrt{2})$.
The maximum value of (\ref{eq-B-7}) subject to (\ref{eq-B-8}), (\ref{eq-B-9}) and (\ref{eq-B-10})
is
\begin{equation}
\frac{-\alpha^6+6\alpha^5-9\alpha^4-4\alpha^3+96\alpha-80}{144(\alpha-1)}\;\mbox{.}
\end{equation}
This value is attained in particular
for $$x_1=x_2=\sigma\mbox{,}\, x_3=1-2\sigma\mbox{,}\, y_1=y_2=1\mbox{ and } y_3=\frac{1+\delta_A(\sigma)}{2}\,\mbox{,} $$
where $\sigma=\frac{5-(\alpha-1)^2}{12}=\frac{-\alpha^2+2\alpha+4}{12}\in (1/4,1/3)$ and
$\delta_A(\sigma)$ is defined in (\ref{eq-dA}).
\end{proposition}

The proof of Proposition~\ref{prop-max} can be obtained using any symbolic mathematical computation program. Nevertheless, we include the proof in the Appendix.

\subsection{The convex regime of $g_t$}

In this subsection, we deal with graphs having the triangle density between $1/9$ and $1/4$.
In order to bound $d_1$ as a function of $d_3$ for $d_3\in [1/9,1/4)$, we study a certain optimization problem involving $d_0$, $d_1$, $d_2$ and $d_3$.
We analyze this problem using the result of Pikhurko and Razborov~\cite{bib-pikhurko16+},
which characterizes the extremal configurations for the edge vs.~triangle density problem.
In particular, we show that the optimal values for the problem
correspond to $3$-vertex graph densities in the complements of the extremal $3$-partite complete graphs.

\begin{lemma}
\label{lm-curve-C}
Let $(d_1,d_3)$ be a point contained in $S_{13}$.
If $d_3\in [1/9,1/4)$, then $d_1\le g_t(d_3)$.
\end{lemma}

\begin{proof}
Let $h(x):[0,1]\to [0,1]$ be the function equal to $g_R(x)$ for $x\in [0,2/3]$ and to $x(2x-1)$ for $x\in [2/3,1]$.
Observe that $h(x)\le g_R(x)$ for every $x\in [0,1]$.
This inequality can be established by a direct computation.
A less technical argument is the following:
Goodman's bound asserts the asymptotic lower bound $d_e(2d_e-1)$ on the triangle density in a graph with edge density $d_e$,
which must be smaller than or equal to the tight asymptotic lower bound $g_R(d_e)$.

Fix $d_3\in [1/9,1/4)$ and consider the problem to maximize $d_1$ subject to
\begin{eqnarray}
d_0+d_1+d_2+d_3 & = & 1 \label{eq-C-1} \\
h\left(\frac{3d_0+2d_1+d_2}{3}\right) & \le & d_0 \label{eq-C-2}
\end{eqnarray}
where $d_0,d_1,d_2\ge 0$. Let $m(d_3)$ be this maximum.

Let $(d_0,d_1,d_2,d_3)\in S$, and
let $(G_n)_{n\in\NN}$ be a convergent sequence of graphs with the limit density of $H_k$ equal to $d_k$, $k=0,1,2,3$.
The values $d_0$, $d_1$, $d_2$ and $d_3$ satisfy (\ref{eq-C-1}) by Proposition~\ref{prop-unit}.
By applying Proposition~\ref{prop-edge} and Theorem~\ref{thm-e-tr} to the limit densities of the complements of the graphs $G_n$,
we obtain that the values $d_0$, $d_1$, $d_2$ and $d_3$ satisfy that
$$g_R\left(\frac{3d_0+2d_1+d_2}{3}\right)\le d_0\;\mbox{.}$$
Since $h(x)\le g_R(x)$ for every $x\in [0,1]$, it follows that
$d_0$, $d_1$, $d_2$ and $d_3$ also satisfy the inequality (\ref{eq-C-2}).
Since the values $d_0$, $d_1$ and $d_2$ are non-negative and they satisfy both (\ref{eq-C-1}) and (\ref{eq-C-2}),
it follows that $d_1\le m(d_3)$.
In the rest of the proof, we will show $m(d_3)=g_t(d_3)$,
which will imply the statement of the lemma.

Recall that $d_3\in [1/9,1/4)$ is fixed.
We first show that $m(d_3)\ge g_t(d_3)$.
Since the functions $h_{B,1}$ and $h_{B,3}$ were defined 
as the asymptotic co-cherry and triangle densities of graphs with a particular structure,
there exists a point $(d_0,d_1,d_2,d_3)$ in $S$ with $d_1=g_t(d_3)=h_{B,1}(h_{B,3}^{-1}(d_3))$.
The values $d_0$, $d_1$ and $d_2$ satisfy both (\ref{eq-C-1}) and (\ref{eq-C-2}).
Hence, it holds that $m(d_3)\ge g_t(d_3)$.

We next establish the opposite inequality, i.e., $m(d_3)\le g_t(d_3)$.
Let $d_0$, $d_1$ and $d_2$ be non-negative reals that satisfy $d_1=m(d_3)$ and the equations (\ref{eq-C-1}) and (\ref{eq-C-2}),
i.e., $d_0$, $d_1$ and $d_2$ form an optimal solution of the considered maximization problem.
Suppose first that (\ref{eq-C-2}) is not satisfied with equality.
If $d_0$ or $d_2$ were positive,
we could decrease this variable by some $\varepsilon>0$ and increase $d_1$ by $\varepsilon$ in such a way that
(\ref{eq-C-2}) still holds.
Since $d_1=m(d_3)$, this cannot be the case and we get that both $d_0$ and $d_2$ are zero.
The equality (\ref{eq-C-1}) now implies that $d_1=1-d_3>3/4$.
Hence, we get that $2d_1/3>1/2$ and that the left hand side of (\ref{eq-C-2}) must be positive,
which is impossible since $d_0=0$.
We conclude that the values $d_0$, $d_1$ and $d_2$ satisfy (\ref{eq-C-2}) with equality.

We next distinguish three cases depending on the value of $d_0$.
If $d_0=0$, then we get that $d_1+d_2=1-d_3$ from (\ref{eq-C-1}) and that $2d_1/3+d_2/3\le 1/2$ from (\ref{eq-C-2}).
The maximum $d_1$ satisfying $d_1+d_2=1-d_3$ and $2d_1/3+d_2/3\le 1/2$ is equal to $d_3+1/2$.
Since $d_3+1/2<g_t(d_3)$ for $d_3\in [1/9,1/4)$, the triple $d_0$, $d_1$ and $d_2$ cannot form an optimal solution.
The inequality $d_3+1/2<g_t(d_3)$ can be derived as follows.
The derivatives of the functions $h_{B,1}$ and $h_{B,3}$ are the following:
\begin{eqnarray*}
h'_{B,1}(\sigma) & = &  6-36\sigma+54\sigma^2\\
h'_{B,3}(\sigma) & = & -6+24\sigma-18\sigma^2
\end{eqnarray*}
This implies that the derivative of $g_t(x)$ at a point $x\in (1/9,1/4)$ is equal to
$$g'_t(x)=\frac{6-36\sigma+54\sigma^2}{-6+24\sigma-18\sigma^2}=1 - \frac{2-4\sigma}{1-\sigma}\;,$$
where $\sigma=h_{B,3}^{-1}(x)$.
Since $\sigma\in (1/3,1/2)$, it follows that $g'_t(x)<1$ for every $x\in (1/9,1/4)$.
Finally, since the function $g_t(x)$ is continuous and for $x=1/4$ we have $g_t(x)=x + 1/2$,
we get that $g_t(x) > x + 1/2$ for every $x\in [1/9,1/4)$.

The next case that we analyze is that $d_0\in (0,2/9]$.
Since the inequality (\ref{eq-C-2}) holds with equality,
we get that
$$d_0+\frac{2d_1}{3}+\frac{d_2}{3}=h^{-1}(d_0)\;\mbox{,}$$
where $h^{-1}$ is the inverse of the function $h$ restricted to the interval $[1/2,1]$.
We derive using (\ref{eq-C-1}) that
$$d_1=d_3-1+3h^{-1}(d_0)-2d_0\;\mbox{.}$$
We now investigate the derivative of the function $g_R$ on the interval $[1/2,2/3]$,
which has values as described in (\ref{eq-g3}) on this interval:
$$g'_R(x)=\left(\frac{(1-\sqrt{4-6x})(2+\sqrt{4-6x})^2}{18}\right)'=1+\sqrt{1-\frac{3x}{2}}\;\mbox{.}$$
Since $h(x)=g_R(x)$ for $x\in [1/2,2/3]$,
we get that the derivative of $h(x)$ on $(1/2,2/3)$ is strictly between $1$ and $3/2$.
Hence, the derivative of $h^{-1}(x)$ on $(0,2/9)$ is strictly between $2/3$ and $1$,
which implies that $3h^{-1}(d_0)-2d_0$ is an increasing function of $d_0$ on $[0,2/9]$.
Since we are considering a triple $d_0$, $d_1$ and $d_2$ maximizing $d_1$,
it must hold that $d_0=2/9$ or $d_2=0$ (if $d_0<2/9$ and $d_2>0$, we could decrease $d_2$ while increasing $d_0$ and so $d_1$).

Suppose that $d_2>0$. It follows that $d_0=2/9$, and
we derive from (\ref{eq-C-1}) and from $d_3\ge 1/9$ that $d_1<1-d_0-d_3\le 2/3$.
Since $m(d_3)\ge 2/3$, such a triple $d_0$, $d_1$ and $d_2$ cannot form an optimal solution.
Hence, it must hold that $d_2=0$.
Since (\ref{eq-C-2}) holds with equality and $d_2=0$,
the result of Pikhurko and Razborov~\cite{bib-pikhurko16+}
on the asymptotic structure of graphs with a given edge density that
minimize the triangle density (see Section~\ref{sec-tr})
implies that $d_0$, $d_1$, $d_2$ and $d_3$ are equal to densities in the complement of a graph
formed by three cliques, where the two largest cliques have the same size.
Since the common size of the two largest cliques uniquely determines $d_3$,
it follows that $d_1=h_{B,1}(h_{B,3}^{-1}(d_3))=g_t(d_3)$.

The final case that remains to be analyzed is that $d_0>2/9$.
Since (\ref{eq-C-2}) holds with equality, we again get that
$$d_0+\frac{2d_1}{3}+\frac{d_2}{3}=h^{-1}(d_0)\;\mbox{,}$$
where $h^{-1}$ is the inverse of the function $h$ restricted to the interval $[1/2,1]$.
We express $d_1$ and substitute to (\ref{eq-C-1}) to get that
\begin{equation}
d_2=2-2d_3+d_0-3h^{-1}(d_0)\;\mbox{.}\label{eq-C-d2}
\end{equation}
Since the derivative of $h(x)$ on the interval $(2/3,1)$ is equal to $4x-1$,
i.e., the derivative $h'(x)$ is less than $3$ for $x\in (2/3,1)$,
we obtain that $x-3h^{-1}(x)$ is a strictly decreasing function of $x \in (2/9,1)$.
In particular, it holds that
$$d_0-3h^{-1}(d_0)<\frac{2}{9}-3h^{-1}\left(\frac{2}{9}\right)=-\frac{16}{9}\;\mbox{.}$$
Since $d_3\ge 1/9$, we obtain from (\ref{eq-C-d2}) that
$$d_2=2-2d_3+d_0-3h^{-1}(d_0)<2-\frac{2}{9}-\frac{16}{9}=0\mbox{,}$$
which is impossible.
This finishes the analysis of the optimal value $m(d_3)$ and
we can now conclude that $m(d_3)=g_t(d_3)$ as desired.
\end{proof}

\subsection{The projection $S_{13}$}

We are now ready to determine the projection $S_{13}$,
which is visualized in Figure~\ref{fig-tr-cc}.

\begin{theorem}
\label{thm-tr-cc}
The projection $S_{13}$ consists precisely of the points $(d_1,d_3)$ such that
$0\le d_3\le 1$ and $0\le d_1\le g_t(d_3)$.
\end{theorem}

\begin{proof}
Let $T$ be the set of the points $(d_1,d_3)$ that satisfy $0\le d_3\le 1$ and \hbox{$0\le d_1\le g_t(d_3)$}.
We first show that $S_{13}\subseteq T$.
Let $(d_1,d_3)\in S_{13}$.
If $d_3\in [0,1/16]$, then $d_1\le g_t(d_3)=3d_3+3/8$ by Lemma~\ref{lm-curve-A}.
If $d_3\in (1/16,1/9)$, then $d_1\le g_t(d_3)$ by Lemma~\ref{lm-curve-B}.
Finally, if $d_3\in [1/9,1/4)$, then $d_1\le g_t(d_3)$ by Lemma~\ref{lm-curve-C}, and
if $d_3\in [1/4,1]$, then $d_1\le 1-d_3=g_t(d_3)$ by Proposition~\ref{prop-unit}.
We conclude that $S_{13}$ is a subset of $T$.

We will next show that two particular sets $T_1$ and $T_2$ are subsets of $S_{13}$.
The set $T_1\subseteq\RR^2$ is the union of the segments with end-points $(0,x)$ and $(g_t(x),x)$ for $x \in [0,1/4]$; 
the set $T_2\subseteq\RR^2$ is the convex hull of the points $(0,1/4)$, $(3/4,1/4)$ and $(0,1)$.
Note that $T_1\cup T_2=T$.

Let $G_0(n,x)$ for every $x\in [-1/4,1/4]$ be the following $n$-vertex graph.
If $x\in [-1/4,0)$, then the vertices of $G_0(n,x)$ are split into five parts $A$, $B$, $C$, $D$ and $E$ such that
$|A|=|B|=|C|=|D|=\lfloor (1/4+x)n\rfloor$ and the remaining vertices belong to $E$.
The graph $G_0(n,x)$ contains all edges between the sets $A$ and $B$,
all edges between the sets $C$ and $D$, and no other edges.

If $x\in [0,1/16)$, then the graph $G_0(n,x)$ is the following random $n$-vertex graph.
Its vertices are split into four parts $A$, $B$, $C$ and $D$ such that
the sizes of any two of the parts differ by at most one.
Two vertices inside the same part are joined by an edge with probability $16x$.
A pair of vertices from the parts $A$ and $B$, respectively, is joined by an edge with probability $1-16x$;
likewise, a pair of vertices from the parts $C$ and $D$, respectively, is joined by an edge with probability $1-16x$.

If $x\in [1/16,1/9)$, let $\sigma=h_{A,3}^{-1}(x)$.
The graph $G_0(n,x)$ has the vertices split into four parts $A$, $B$, $C$ and $D$ such that
$|A|=|B|=\lfloor (1-\sigma)n/2\rfloor$ and the remaining vertices are split among $C$ and $D$ in such a way that
the sizes of $C$ and $D$ differ by at most one.
Each vertex of $A$ is adjacent to exactly $\lfloor \delta_A(\sigma)|B|\rfloor$ vertices of $B$, and
each vertex of $B$ is adjacent to exactly $\lfloor \delta_A(\sigma)|A|\rfloor$ vertices of $A$,
where $\delta_A(\sigma)$ is as in (\ref{eq-dA});
since $\delta_A(\sigma)\in [0,1]$ and $|A|=|B|$, this is possible.
In addition, all pairs of vertices inside each of the four parts are joined by edges.

Finally, if $x\in [1/9,1/4]$, then let $\sigma=h_{B,3}^{-1}(x)$.
The graph $G_0(n,x)$ has the vertices split into three parts $A$, $B$ and $C$ such that
$|A|=|B|=\lfloor \sigma n\rfloor$, and
two vertices are joined by an edge if they belong to the same part.

For every $x\in [-1/4,1/4]$, the sequence $(G_0(n,x))_{n\in\NN}$ is convergent with probability one and
the limit triangle density is $\max\{0,x\}$, and
the limit co-cherry density is $24\left(\frac{1}{4}+x\right)^2\left(\frac{1}{4}-x\right)$ if $x\in [-1/4,0)$, and
it is equal to $g_t(x)$ if $x\in [0,1/4]$.

We next define a graph $G_1(n,a,x)$ for $a\in [0,1]$ and $x\in [-1/4,1/4]$ to be the graph
obtained from the graph $G_0(\lceil (1-a)n\rceil,x)$ by adding $\lfloor an\rfloor$ vertices that
are adjacent to all vertices of the graph $G_1(n,a,x)$.
The sequence of graphs $(G_1(n,a,x))_{n\in\NN}$ is convergent with probability one
for every $a\in [0,1]$ and \hbox{$x\in [-1/4,1/4]$};
let $h_1(a,x)\in\RR^2$ be the pair formed by the limit co-cherry density and the limit triangle density.

Let $t_1\subseteq\RR^2$ be the convex hull of the points $(0,-1/4)$, $(0,1/4)$, $(1,1/4)$ and $(1,-1/4)$.
Observe that the following holds for all $a\in [0,1]$ and $x\in [0,1/4]$:
\begin{eqnarray*}
h_1(0,-x) & = & \left(24\left(\frac{1}{4}+x\right)^2\left(\frac{1}{4}-x\right),\; 0\right) \\
h_1(0,x) & = & (g_t(x),x) \\
h_1(a,1/4) & = & \left(\frac{3}{4}(1-a)^3,\; 1-\frac{3}{4}(1+a)(1-a)^2\right) \\
h_1(1,x) & = & (0,\; 1) \\
h_1(1,-x) & = & (0,\; 1) \\
h_1(a,-1/4) & = & (0,\; a^3+3a^2(1-a)) 
\end{eqnarray*}
Since the derivative of $1-\frac{3}{4}(1+a)(1-a)^2$ is equal to
$$\frac{3}{4}(1-a)(1+3a) \;\mbox{,}$$
which is positive for $a\in (0,1)$,
we conclude that the closed region of $\RR^2$ bounded by the image of the boundary of $t_1$ contains the set $T_1$.
Since $h_1$ is a continuous map from $t_1$ to $\RR^2$ and $t_1$ is a topological $2$-disc,
it follows that $T_1\subseteq h_1(t_1)$.
Finally, since $h_1(t_1)$ is a subset of $S_{13}$, $T_1$ is also a subset of $S_{13}$.

It remains to show that $T_2$ is also a subset of $S_{13}$.
Define $G_2(n,a,p)$ to be the random $n$-vertex graph that is obtained as we now describe.
Split the vertices into two sets $A$ and $B$ with $\lfloor an\rfloor$ and $\lceil (1-a)n\rceil$ vertices, respectively.
Join all vertices in the set $A$ by edges,
join two vertices in different sets by an edge with probability $p$, and
join two vertices in the set $B$ by an edge with probability $1-p$.
The sequence $(G_2(n,a,p))_{n\in\NN}$ converges with probability one for all $a,p\in [0,1]$.
The following are the limit co-cherry and triangle densities in this sequence.
\begin{eqnarray*}
& & 3(1-a)^3p^2(1-p)+3a(1-a)^2((1-p)^3+2p^2(1-p))+3a^2(1-a)(1-p)^2 \\
& & (1-a)^3(1-p)^3+3a(1-a)^2p^2(1-p)+3a^2(1-a)p^2+a^3 
\end{eqnarray*}
Let $h_2(a,p)\in\RR^2$ be these limit density of co-cherries and triangles.

Let $t_2\subseteq\RR^2$ be the convex hull of the points $(0,1)$, $(1/2,0)$, $(1,0)$ and $(1,1)$, and
observe that the following holds:
\begin{eqnarray*}
h_2\left(a,1-2a\right) & = & \left(6 a (8 a^5-18 a^4+7 a^3+7 a^2-5 a+1), \right. \\
& & \qquad\qquad\qquad \left. a^2 (16 a^4-60 a^3+78 a^2-42 a+9) \right) \\
h_2(a,0) & = & \left(3a(1-a),\; 1-3a(1-a)\right) \\
h_2(1,p) & = & \left(0,\; 1\right) \\
h_2(a,1) & = & \left(0,\; a^2(3-2a)\right) 
\end{eqnarray*}
Since the derivative of $a^2 (16 a^4-60 a^3+78 a^2-42 a+9)$ is equal to
$$6a(1 - a)(3 - 18 a + 34 a^2 - 16 a^3)\;\mbox{,}$$
which is positive for $a\in (0,1/2)$,
we conclude that the closed region of $\RR^2$ bounded by the image of the boundary of $t_2$
contains $T_2$.
Since $h_2$ is a continuous map from $t_2$ to $\RR^2$ and $t_2$ is a topological $2$-disc,
it follows that $T_2\subseteq h_2(t_2)$.
Finally, since $h_2(t_2)$ is a subset of $S_{13}$, $T_2$ is a subset of $S_{13}$,
which finishes the proof of the theorem.
\end{proof}

\section{Cherry vs.~co-cherry projection}
\label{sec-ch-cc}

In this section, we determine the projection $S_{12}$,
which turned out to be the easiest among the three projections $S_{12}$, $S_{13}$ and $S_{23}$.
The projection is visualized in Figure~\ref{fig-ch-cc}.

\begin{figure}[ht]
\begin{center}
\epsfbox{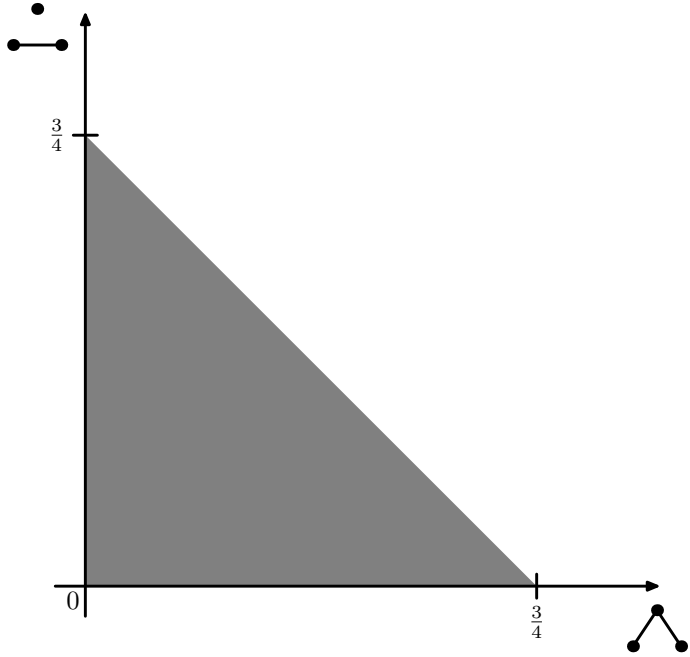}
\end{center}
\caption{Possible densities of cherries and co-cherries in graphs.}
\label{fig-ch-cc}
\end{figure}

\begin{theorem}
\label{thm-ch-cc}
The projection $S_{12}$ consists precisely of the points $(d_1,d_2)$ such that
$d_1\ge 0$, $d_2\ge 0$ and $d_1+d_2\le 3/4$.
\end{theorem}

\begin{proof}
Let $T$ be the set of the points $(d_1,d_2)$ that satisfy $d_1\ge 0$, $d_2\ge 0$ and $d_1+d_2\le 3/4$.
We aim to show that $S_{12}=T$.
Consider a point $(d_0,d_1,d_2,d_3)\in S$.
Since $d_0+d_3\ge 1/4$ by Goodman's bound,
it follows from Proposition~\ref{prop-unit} that $d_1+d_2\le 3/4$.
Hence, the projection $S_{12}$ is a subset of the set $T$.

We now define an $n$-vertex random graph $G(n,a,p)$ where $a,p\in [0,1]$ as follows:
split the vertices of $G$ into a set $A$ of size $\lfloor an\rfloor$ and a set $B$ of size $\lceil (1-a)n\rceil$,
a pair of vertices inside the set $A$ or inside the set $B$ is joined by an edge with probability $p$, and
the remaining pairs of vertices are joined with probability $1-p$.
The sequence of random graphs $(G(n,a,p))_{n\in\NN}$ is convergent with probability one.
Observe that the expected co-cherry density in $G(n,a,p)$ is equal to
$$(a^3+(1-a)^3)\cdot 3p(1-p)^2 +(3a(1-a))\cdot (p^3+2(1-p)^2p)=$$
\begin{equation}
3p(1-p)^2+3a(1-a)p(2p-1)
\label{eq-ch-cc-CC}
\end{equation}
and the expected cherry density is equal to
$$(a^3+(1-a)^3)\cdot 3p^2(1-p)+(3a(1-a))\cdot ((1-p)^3+2p^2(1-p))=$$
\begin{equation}
3p^2(1-p)+3a(1-a)(1-p)(1-2p)\;\mbox{.}
\label{eq-ch-cc-CR}
\end{equation}
Let $h(a,p)\in\RR^2$ be the limit co-cherry and cherry densities in the sequence $(G(n,a,p))_{n\in\NN}$;
the standard concentration arguments yield that the sequence converges and
the coordinates of $h(a,p)$ are equal to (\ref{eq-ch-cc-CC}) and (\ref{eq-ch-cc-CR}) with probability one, respectively.
Note that $h(a,p)\in S_{12}$ for all choices $a,p\in [0,1]$.

Let $T_1$ be the subset of $T$ formed by the points $(d_1,d_2)\in T$ with $d_1\ge d_2$, and
let $T_2$ be the subset formed by the points $(d_1,d_2)\in T$ with $d_1\le d_2$.
Further, let $t_1\subseteq\RR^2$ be the set of points $(x,y)$ such that $0\le x\le 1/2$ and $0\le y\le x$, and
let $t_2\subseteq\RR^2$ be the set of points $(x,y)$ such that $1/2\le x\le 1$ and $x\le y\le 1$.
Note that $t_1$ is the convex hull of the points $(0,0)$, $(1/2,0)$ and $(1/2,1/2)$, and
$t_2$ is the convex hull of the points $(1/2,1/2)$, $(1/2,1)$ and $(1,1)$.
Observe that the following holds:
\begin{eqnarray*}
h(a,0) & = & (0,\; 3a(1-a)) \\
h(1/2,p) & = & \left(\frac{3}{8}-\frac{3}{8}(1-2p)^3,\; \frac{3}{8}+\frac{3}{8}(1-2p)^3\right) \\
h(a,a) & = & \left(3a(1-a)\left(1-2a+2a^2\right),\; 3a(1-a)\left(1-2a+2a^2\right)\right)
\end{eqnarray*}
Hence, the boundary of the triangle $t_1$ is mapped by $h$ to the boundary of $T_1$.
Since $h$ is a continuous map from $t_1$ to $\RR^2$,
$t_1$ is a topological $2$-disc and its boundary is mapped to the boundary of $T_1$,
it follows $T_1\subseteq h(t_1)$.
Since $h(t_1)$ is a subset of $S_{12}$, it follows that $T_1\subseteq S_{12}$.
The analogous argument yields that
the boundary of triangle $t_2$ is mapped by $h$ to the boundary of $T_2$,
which implies that $T_2\subseteq h(t_2)$ and thus $T_2\subseteq S_{12}$.
We conclude that all the points of $T=T_1\cup T_2$ are contained in $S_{12}$,
which finishes the proof that $S_{12}=T$.
\end{proof}

\section{Triangle vs.~cherry projection}
\label{sec-tr-ch}

In this section, we determine the last remaining projection $S_{23}$.
Recall that $g_R^{-1}:[0,1]\to [1/2,1]$
is the inverse of the function $g_R$ from Theorem~\ref{thm-e-tr} restricted to the interval $[1/2,1]$.

\begin{figure}[ht]
\begin{center}
\epsfbox{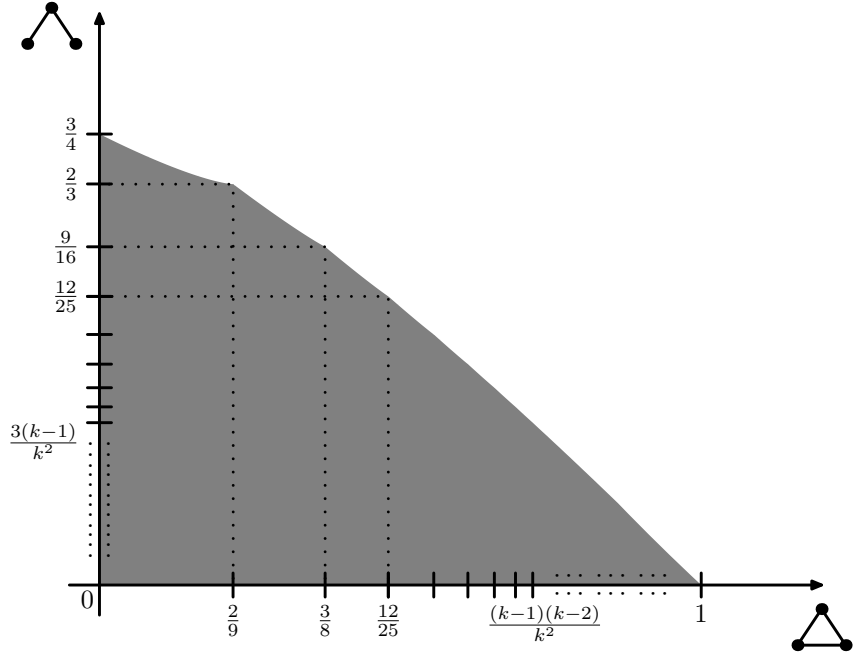}
\end{center}
\caption{Possible densities of triangles and cherries in graphs.}
\label{fig-tr-ch}
\end{figure}

\begin{theorem}
\label{thm-tr-ch}
The projection $S_{23}$ consists precisely of the points $(d_2,d_3)$ such that
$d_2\ge 0$, $d_3\ge 0$ and $d_2\le \frac{3}{2}\left(g_R^{-1}(d_3)-d_3\right)$.
\end{theorem}

\begin{proof}
Let $T$ be the set of the points $(d_2,d_3)$ that satisfy the inequalities $d_2\ge 0$, $d_3\ge 0$ and $d_2\le \frac{3}{2}\left(g_R^{-1}(d_3)-d_3\right)$.
We will show that $S_{23}=T$.
Consider a point $(d_0,d_1,d_2,d_3)\in S$.
Clearly, $d_1$, $d_2$ and $d_3$ are non-negative.
By Theorem~\ref{thm-e-tr}, we have that
$$d_3\ge g_R\left(\frac{d_1+2d_2+3d_3}{3}\right)\;\mbox{.}$$
It follows that
$$\frac{2d_2+3d_3}{3}\le\frac{d_1+2d_2+3d_3}{3}\le g_R^{-1}(d_3)\;\mbox{,}$$
which yields that
$$d_2\le\frac{3}{2}\left(g_R^{-1}(d_3)-d_3\right)\;\mbox{.}$$
We conclude that the projection $S_{23}$ is a subset of the set $T$.

We now define an $n$-vertex graph $G(n,a,b)$ for $a\in [0,1/2]$ and $b\in [0,1]$.
The graph $G(n,a,b)$ has $\lfloor (1-b)n\rfloor$ isolated vertices.
If $a=0$, the remaining $\lceil bn\rceil$ vertices form a complete graph.
Otherwise, the remaining $\lceil bn\rceil$ vertices form a complete multipartite graph
with $\lfloor a^{-1}\rfloor+1$ parts such that
$\lfloor a^{-1}\rfloor$ parts have size $\lfloor abn\rfloor$ and
the remaining part has size $\lceil bn\rceil-\lfloor a^{-1}\rfloor\cdot\lfloor abn\rfloor$.
One or more of the parts of the complete multipartite graph can be empty;
this happens if either $bn$ is an integer, $a^{-1}$ is an integer and $a^{-1}$ divides $bn$, or $abn<1$.
It is straightforward to show that
the sequence of graphs $(G(n,a,b))_{n\in\NN}$ converges for every fixed pair $a\in [0,1/2]$ and $b\in [0,1]$.
Let $h(a,b)\in\RR^2$ be the limit cherry and triangle densities in the sequence $(G(n,a,b))_{n\in\NN}$.
Observe that $h(a,b)$ is a continuous function from $[0,1/2]\times [0,1]$ to $\RR^2$.

We now investigate the function $h(a,b):[0,1/2]\times [0,1]\to\RR^2$.
First observe that $h(a,b)=b^3\cdot h(a,1)$.
We have $h(0,1)=(0,1)$ and thus $h(0,b)=(0,b^3)$.
Fix $a\in (0,1/2]$ and consider the sequence of graphs $G(n,a,1)$.
Let $d_e$ be the limit edge density and let $d_k$ be the limit density of $H_k$, the $k$-edge $3$-vertex graph.
Note that $d_1=0$ since each of the graphs $G(n,a,1)$ is a complete multipartite graph.
Next observe that the graph $G(n,a,1)$ is a complete multipartite graph with $\lfloor a^{-1}\rfloor+1$ parts such that
the fraction of the vertices contained in $\lfloor a^{-1}\rfloor$ of its parts converges to $a$.
Hence, it follows that $d_3=g_R(d_e)$.
The construction of graphs $G(n,a,1)$ and the fact that $a\le 1/2$, implies that $d_e\ge 1/2$.
Since $d_e=(d_1+2d_2+3d_3)/3$ (see Proposition~\ref{prop-edge}) and $d_1=0$, we get the following:
$$g_R^{-1}(d_3)=\frac{2d_2+3d_3}{3}\;\mbox{.}$$
Thus $d_2=\frac{3}{2}\left(g_R^{-1}(d_3)-d_3\right)$.
We conclude that the points $h(a,1)$, $a\in [0,1/2]$, form the curve of the boundary of $T$
between the points $(0,1)$ and $(3/4,0)$.

Consider now the square $t=[0,1/2]\times [0,1]$.
We claim that the boundary of $t$ is mapped by $h$ to the boundary of $T$.
Indeed, the segment $\{1/2\}\times [0,1]$ is mapped to the segment $\{0\}\times [0,3/4]$,
the segment $[0,1/2]\times\{0\}$ is mapped to the point $(0,0)$,
the segment $\{0\}\times [0,1]$ is mapped to the segment $[0,1]\times\{0\}$, and
the segment $[0,1/2]\times\{1\}$ is mapped to the remaining part of the boundary of~$T$.
Since $h$ is a continuous map from $t$ to $\RR^2$,
$t$ is a topological $2$-disc and its boundary is mapped to the boundary of $T$,
it follows $T\subseteq h(t)$.
Since $h(t)$ is a subset of $S_{23}$, we conclude that $T\subseteq S_{23}$,
which yields that $T=S_{23}$.
\end{proof}

\section{Boundaries of the projections}
\label{sec-concl}

In this section,
we briefly discuss the structure of graphs on the boundaries of the projections $S_{13}$, $S_{12}$ and $S_{23}$.
Those on the parts of the boundaries with the zero density of one of the graphs $H_1$, $H_2$ and $H_3$
can be any graphs with corresponding zero density.
While the structure of triangle-free graphs can be very complex,
every graph with the zero density of $H_2$ is a union of cliques;
graphs with the zero density of $H_1$ are then their complements.
The situation is more interesting for the other parts of the boundaries of the projections.

Let us start with the projection $S_{13}$.
We will describe the structure of graphons $W$ with the co-cherry density equal to $g_t(d(H_3,W))$.
We distinguish four cases based on the triangle density of $W$;
each case corresponds to one of the smooth parts of the curve $(x,g_t(x))$, $x\in [0,1]$.
If the triangle density belongs to the interval $[0,1/16]$,
then the upper bound follows from Lemma~\ref{lm-curve-A}.
The equality $d(H_1,W)=g_t(d(H_3,W))$ holds
if and only if (\ref{eq-A-1}) in the proof of Lemma~\ref{lm-curve-A} holds with equality.
Hence, a graphon $W$ satisfies $d(H_1,W)=g_t(d(H_3,W))$ and $d(H_3,W)\in [0,1/16]$
if and only if $d_W(x)=1/4$ for almost every $x\in [0,1]$.

If the triangle density belongs to the interval $(1/16,1/9)$,
then the proof of Lemma~\ref{lm-curve-B} yields that
a graphon $W$ satisfies $d(H_1,W)=g_t(d(H_3,W))$ if and only if
it has three components of measures $x_1$, $x_2$ and $x_3$ with relative edge densities $y_1$, $y_2$ and $y_3$ and
zero co-triangle densities such that
the values of $x_1$, $x_2$, $x_3$, $y_1$, $y_2$ and $y_3$
maximize the sum (\ref{eq-B-7}) subject to (\ref{eq-B-8}), (\ref{eq-B-9}) and (\ref{eq-B-10}).
Proposition~\ref{prop-max} yields that such values are unique up to their permutation.
We conclude that a graphon satisfies $d(H_1,W)=g_t(d(H_3,W))$ and $d(H_3,W)\in (1/16,1/9)$
if and only if it is a limit of the sequence $(G_0(n,x))_{n\in\NN}$ for some $x\in [1/16,1/9]$,
where the graph $G_0(n,x)$ is defined as in the proof of Theorem~\ref{thm-tr-cc}.

If the triangle density belongs to the interval $[1/9,1/4)$,
then the proof of Lemma~\ref{lm-curve-C} implies that
the density of cherries is zero and
the density of co-triangles is equal to $g_R(1-d(K_2,W))$,
i.e., the graphon $1-W$ is one of the graphons minimizing the triangle density for a given edge density.
The structure of such graphons was determined by Pikhurko and Razborov~\cite{bib-pikhurko16+}.
It follows that a graphon $W$ satisfies that $d(H_1,W)=g_t(d(H_3,W))$ and $d(H_3,W)\in [1/9,1/4)$
if and only if it is a limit of the sequence $(G_0(n,x))_{n\in\NN}$ for some $x\in [1/9,1/4)$,
where the graph $G_0(n,x)$ is defined as in the proof of Theorem~\ref{thm-tr-cc}.

Finally, if the triangle density belongs to the interval $[1/4,1]$,
then Proposition~\ref{prop-unit} gives that the co-triangle and cherry density must be zero,
i.e., such graphons corresponds to unions of two complete graphons.
Consequently, a graphon $W$ satisfies that $d(H_1,W)=g_t(d(H_3,W))$ and $d(H_3,W)\in [1/4,1]$
if and only if it is a limit of the sequence $(G_2(n,a,0))_{n\in\NN}$ for some $a\in [0,1/2]$,
where the graph $G_2(n,a,p)$ is defined as in the proof of Theorem~\ref{thm-tr-cc}.

The situation is less complex for the projections $S_{12}$ and $S_{23}$.
The case of the projection $S_{12}$ is quite simple:
the structure of extremal configurations for Goodman's bound implies that
a graphon $W$ satisfies \hbox{$d(H_1,W)+d(H_2,W)=3/4$} if and only if $d_W(x)=1/2$ for almost every $x\in [0,1]$.
In the case of the projection $S_{23}$, we need to inspect the proof of Theorem~\ref{thm-tr-ch}.
We note that the equality in the last inequality in the statement of the theorem holds if and only if the graphon
has zero co-cherry density and it is one of the extremal configurations described
by Pikhurko and Razborov in~\cite{bib-pikhurko16+}.
It follows that a graphon $W$ satisfies that $d(H_2,W)=\frac{3}{2}\left(g_R^{-1}(d(H_3,W))-d(H_3,W)\right)$ if and only if
either it is equal to one almost everywhere or
there exists $k$ and $\alpha\in [\frac{1}{k},\frac{1}{k-1})$ and
$W$ is a limit of a sequence of complete $k$-partite graphs with $k-1$ parts containing the fraction $\alpha$ of the vertices.

\section*{Acknowledgements}

The authors would like to thank D\'aniel Kor\'andi and Jordan Venters
for stimulating discussions on the topics covered in this paper.

\newpage
\section*{Appendix}

We rewrite Proposition~\ref{prop-max} in a self-contained form and present a formal proof. 

\setcounter{theorem}{9}
\begin{proposition}
Let $\alpha\in(2,1+\sqrt{2})$.
The maximum value of 
\begin{equation}
F(x_1,x_2,x_3,y_1,y_2,y_3)=\sum_{j=1}^3 x_j^3\left(3-\alpha-3(3-\alpha)y_j-3(\alpha-1)y_j^2\right)+3x_j^2y_j
\label{eq-B-7A}
\end{equation}
subject to 
\begin{eqnarray}
x_1,x_2,x_3 & \ge & 0\,\mbox{,} \label{eq-B-8A}\\
x_1+x_2+x_3 & = & 1\,\mbox{, and} \label{eq-B-9A}\\
y_1,y_2,y_3 & \in & (1/2,1]\,\mbox{.} \label{eq-B-10A}
\end{eqnarray}
is
\begin{equation}
\frac{-\alpha^6+6\alpha^5-9\alpha^4-4\alpha^3+96\alpha-80}{144(\alpha-1)}\;\mbox{.}\label{eq-B-M}
\end{equation}
This value is attained in particular
for $$x_1=x_2=\sigma\mbox{,}\, x_3=1-2\sigma\mbox{,}\, y_1=y_2=1\mbox{ and } y_3=\frac{(1-2\sigma)\sqrt{5-12\sigma} + 4\sigma-1}{2(1-2\sigma)\sqrt{5-12\sigma}}\,\mbox{,} $$
where $\sigma=\frac{5-(\alpha-1)^2}{12}=\frac{-\alpha^2+2\alpha+4}{12}\in (1/4,1/3)$.
\end{proposition}

\begin{proof}
We determine the maximum value of (\ref{eq-B-7A}) when the constraint (\ref{eq-B-10A}) is replaced with
$$y_1,y_2,y_3\in [1/2,1]\,\mbox{,}$$
and show that the maximum value of this relaxed problem is the same.
Before proceeding with the proof, note that if $\alpha\in(2,1+\sqrt{2})$,
then $\sigma\in (1/4,1/3)$.

Fix $\alpha\in(2,1+\sqrt{2})$ and an optimal solution $x_1,\ldots,x_3$ and $y_1,\ldots,y_3$.
By symmetry, we may assume that $x_1\le x_2\le x_3$.
We start by investigating the partial derivative of (\ref{eq-B-7A}) with respect to $y_j$,
which is equal to
\begin{equation}
x_j^3\left(-3(3-\alpha)-6(\alpha-1)y_j\right)+3x_j^2\,\mbox{.}
\label{eq-B-12}
\end{equation}
If $x_j\not=0$ and (\ref{eq-B-12}) is zero, then it holds that
$$y_j=\frac{x_j^{-1}-(3-\alpha)}{2(\alpha-1)}\,\mbox{,}$$
which belongs to $[1/2,1]$ if $x_j\in [(1+\alpha)^{-1},1/2]$.
If $x_j=0$, then the value of (\ref{eq-B-7A}) does not depend on $y_j$ and we can set it arbitrarily.
Hence, we conclude that we can assume that the optimal solution that we have fixed satisfies that
\begin{equation}
y_j=\left\{\begin{array}{cl}
           1 & \mbox{if $x_j\le\frac{1}{\alpha+1}$,} \\
           1/2 & \mbox{if $x_j\ge\frac{1}{2}$, and} \\
           \frac{x_j^{-1}-(3-\alpha)}{2(\alpha-1)} & \mbox{otherwise.}
           \end{array}\right.
\label{eq-B-13}
\end{equation}
If $y_j=1$, then the $j$-th term in (\ref{eq-B-7A}) is equal to $3x_j^2-(\alpha+3)x_j^3$, and
the partial derivative of (\ref{eq-B-7A}) with respect to $x_j$ is
\begin{equation}
6x_j-3(\alpha+3)x_j^2\;\mbox{.} \label{eq-B-14}
\end{equation}
If $y_j=1/2$, then the $j$-th term is $\frac{3}{2}x_j^2-\frac{\alpha+3}{4}x_j^3$, and
the partial derivative of (\ref{eq-B-7A}) with respect to $x_j$ is
\begin{equation}
3x_j-\frac{3}{4}(\alpha+3)x_j^2\;\mbox{.} \label{eq-B-15}
\end{equation}
Finally, if the third case of (\ref{eq-B-13}) applies, then the $j$-th term in (\ref{eq-B-7A}) is 
$$\frac{(3-\alpha)(\alpha+5)x_j^3-6(3-\alpha)x_j^2+3x_j}{4(\alpha-1)}\;\mbox{,}$$
and the partial derivative of (\ref{eq-B-7A}) with respect to $x_j$ is
\begin{equation}
\frac{3(3-\alpha)(\alpha+5)x_j^2-12(3-\alpha)x_j+3}{4(\alpha-1)}\;\mbox{.} \label{eq-B-16}
\end{equation}

We next distinguish three cases depending on how many of the variables $x_1,\ldots,x_3$ are equal to zero.
If two of the variables $x_1,\ldots,x_3$ are zero, i.e., $x_1=x_2=0$, then $x_3=1$.
This implies that $y_3=1/2$ and the value of (\ref{eq-B-7A}) is $\frac{3-\alpha}{4}$,
which is less than (\ref{eq-B-M}).

Suppose that exactly one of the variables is zero, i.e., $x_1=0$. To analyze this case,
we need to distinguish four cases depending on the value of $x_2$.
Note that $x_2\in (0,1/2]$.
\begin{itemize}
\item {\bf The value of $x_2$ belongs to $\left(0,\frac{1}{\alpha+1}\right)$.}
      Note that $x_3\in (1/2,1)$. Using the method of Lagrange multipliers, we get that
      $$\frac{\partial}{\partial x_2}F=\frac{\partial}{\partial x_3}F\;\mbox{.}$$
      It follows from (\ref{eq-B-14}) and (\ref{eq-B-15}) that
      $$6x_2-3(\alpha+3)x_2^2=3x_3-\frac{3}{4}(\alpha+3)x_3^2\;\mbox{.}$$
      This would imply that either $x_2=1/3$ and $x_3=2/3$, or
      $x_2=\frac{1-\alpha}{\alpha+3}$ and $x_3=\frac{2(\alpha+1)}{\alpha+3}$.
      In either of the cases, $x_2$ does not belong to the interval $\left(0,\frac{1}{\alpha+1}\right)$.
\item {\bf The value of $x_2$ is equal $\frac{1}{\alpha+1}$.}
      It follows that $x_3=\alpha/(\alpha+1)$, $y_2=1$ and $y_3=1/2$.
      Consequently, the value of (\ref{eq-B-7A}) is
      $$\frac{-\alpha^4+3\alpha^3+6\alpha^2+8\alpha}{4(\alpha+1)^3}\,\mbox{,}$$
      which is less than (\ref{eq-B-M}).
\item {\bf The value of $x_2$ belongs to $\left(\frac{1}{\alpha+1},1/2\right)$.}
      Applying the method of Lagrange multipliers, we get that
      $$\frac{3(3-\alpha)(\alpha+5)x_2^2-12(3-\alpha)x_2+3}{4(\alpha-1)}=3x_3-\frac{3}{4}(\alpha+3)x_3^2\;\mbox{.}$$
      It follows that
      $$x_2=\frac{\alpha^2-2\alpha+2}{6}\mbox{ and }x_3=\frac{-\alpha^2+2\alpha+4}{6}\;\mbox{.}$$
      The value of (\ref{eq-B-7A}) is then the same as (\ref{eq-B-M}).
      However, since $y_3=1/2$, the values of the variables do not form an optimal solution of the original problem.
\item {\bf The value of $x_2$ is equal to $1/2$.}
      It follows that $x_3=1/2$, $y_2=1/2$ and $y_3=1/2$. Hence, the value of (\ref{eq-B-7A}) is $\frac{9-\alpha}{16}$,
      which is less than (\ref{eq-B-M}).
\end{itemize}

In the rest of the proof, we assume that all the three variables $x_1$, $x_2$ and $x_3$ are positive.
We start by considering the cases when two of the values of $x_1$, $x_2$ and $x_3$
are equal to $\frac{1}{\alpha+1}$ or $1/2$. There are only two cases to analyze (assuming that $x_1\le x_2\le x_3$).
\begin{itemize}
\item {\bf It holds that $x_1=x_2=\frac{1}{\alpha+1}$ and $x_3=\frac{\alpha-1}{\alpha+1}$.}
      It follows that $y_1=y_2=1$ and
      $$y_3=\frac{\alpha^2-3\alpha+4}{2(\alpha-1)^2}\;\mbox{.}$$
      The value of (\ref{eq-B-7A}) is then equal to 
      $$\frac{-\alpha^4+6 \alpha^3+3 \alpha^2-16 \alpha+36}{4 (\alpha+1)^3}\;\mbox{,}$$
      which is smaller than (\ref{eq-B-M}).
\item {\bf It holds that $x_1=\frac{\alpha-1}{2(\alpha+1)}$, $x_2=\frac{1}{\alpha+1}$ and $x_3=\frac{1}{2}$.}
      We get that $y_1=y_2=1$ and $y_3=1/2$.
      This implies that the value of (\ref{eq-B-7A}) is
      $$\frac{-5 \alpha^3+35 \alpha^2-11 \alpha+45}{32 (\alpha+1)^2}\;\mbox{,}$$
      which is also smaller than (\ref{eq-B-M}).
\end{itemize}

We next consider the cases when exactly one of the values of $x_1$, $x_2$ and $x_3$
is equal to $\frac{1}{\alpha+1}$ or $1/2$. Recall that $x_1\le x_2\le x_3$.
\begin{itemize}
\item {\bf The value of $x_1$ belongs to $\left(0,\frac{1}{\alpha+1}\right)$, the value of $x_2$ is equal to $\frac{1}{\alpha+1}$, and
           the value of $x_3$ is smaller than $1/2$.}
      Using the method of Lagrange multipliers, we get that
      $$\frac{\partial}{\partial x_1}F=\frac{\partial}{\partial x_3}F\;\mbox{.}$$
      It follows using (\ref{eq-B-14}) and (\ref{eq-B-16}) that
      $$6x_1-3(\alpha+3)x_1^2=\frac{3(3-\alpha)(\alpha+5)x_3^2-12(3-\alpha)x_3+3}{4(\alpha-1)}\;\mbox{,}$$
      which implies that
\footnotesize
      \begin{eqnarray*}
      x_1 &\! = \!& \frac{-\alpha^4+3 \alpha^3+15 \alpha^2+\alpha-10}{3 (\alpha+1)^4}\pm\frac{\sqrt{ 4 \alpha^6-8 \alpha^5-39 \alpha^4+84 \alpha^3+74 \alpha^2-196 \alpha+97}}{3 (\alpha+1)^3} \\
      x_3 &\! = \!& \frac{4 \alpha^4+6 \alpha^3-6 \alpha^2+2 \alpha+10}{3 (\alpha+1)^4}\mp\frac{\sqrt{ 4 \alpha^6-8 \alpha^5-39 \alpha^4+84 \alpha^3+74 \alpha^2-196 \alpha+97}}{3 (\alpha+1)^3}
      \end{eqnarray*}
\normalsize
      Since $x_3>1/2$ for one of the two choices of the sign, we get that
\footnotesize
      \begin{eqnarray*}
      x_1 &\! = &\! \frac{-\alpha^4+3 \alpha^3+15 \alpha^2+\alpha-10}{3 (\alpha+1)^4}+\frac{\sqrt{ 4 \alpha^6-8 \alpha^5-39 \alpha^4+84 \alpha^3+74 \alpha^2-196 \alpha+97}}{3 (\alpha+1)^3} \\
      x_3 &\! = &\! \frac{4 \alpha^4+6 \alpha^3-6 \alpha^2+2a+10}{3 (\alpha+1)^4}-\frac{\sqrt{ 4 \alpha^6-8 \alpha^5-39 \alpha^4+84 \alpha^3+74 \alpha^2-196 \alpha+97}}{3 (\alpha+1)^3}\;\mbox{,}
      \end{eqnarray*}
\normalsize
      which yields the value of (\ref{eq-B-7A}) smaller than (\ref{eq-B-M}).
\item {\bf The value of $x_1$ belongs to $\left(0,\frac{1}{\alpha+1}\right)$, the value of $x_2$ is equal to $\frac{1}{\alpha+1}$, and
           the value of $x_3$ is larger than $1/2$.}
      We derive using the method of Lagrange multipliers that
      $$6x_1-3(\alpha+3)x_1^2=3x_3-\frac{3}{4}(\alpha+3)x_3^2\;\mbox{.}$$
      Since $x_3>1/2$, it follows that
      \begin{eqnarray*}
      x_1 & = & \frac{-\alpha^2+\alpha+4}{\alpha^2+4\alpha+3} \\
      x_3 & = & \frac{2(\alpha^2+\alpha-2)}{\alpha^2+4\alpha+3} 
      \end{eqnarray*}
      Consequently, the value of (\ref{eq-B-7A}) is
      $$\frac{-\alpha^5+\alpha^4+11\alpha^3+3\alpha^2-6\alpha+24}{ (\alpha+1)^2 (\alpha+3)^2}\,\mbox{,}$$
      which is smaller than (\ref{eq-B-M}).
\item {\bf The value of $x_1$ is equal to $\frac{1}{\alpha+1}$, and
           the remaining values are larger than $\frac{1}{\alpha+1}$.}
      Note that both $x_2$ and $x_3$ must be less than $1/2$.	   
      Since the partial derivatives of $F$ with respect to $x_2$ and $x_3$ are equal,
      it follows that $x_2=x_3$. Hence, it holds that
      $$x_2=x_3=\frac{\alpha}{2(\alpha+1)}\;\mbox{.}$$
      Consequently, the value of (\ref{eq-B-7A}) is
      $$\frac{ - \alpha (\alpha^4-10 \alpha^3-3 \alpha^2-20 \alpha+20)}{16 (\alpha-1) (\alpha+1)^3}\,\mbox{,}$$
      which is smaller than (\ref{eq-B-M}).
\item {\bf Both $x_1$ and $x_2$ belong to $\left(0,\frac{1}{\alpha+1}\right)$, and $x_3=1/2$.}
      Since the partial derivatives of $F$ with respect to $x_1$ and $x_2$ must be the same,
      it follows that $x_1=x_2=1/4$.
      The value of (\ref{eq-B-7A}) is then equal to $\frac{9-\alpha}{16}$,
      which is again smaller than (\ref{eq-B-M}).
\item {\bf The value of $x_1$ belongs to $\left(0,\frac{1}{\alpha+1}\right)$,
           the value of $x_2$ belongs to $\left(\frac{1}{\alpha+1},1/2\right)$, and $x_3=1/2$.}
      We derive using the method of Lagrange multipliers that
      $$6x_1-3(\alpha+3)x_1^2=\frac{3(3-\alpha)(\alpha+5)x_2^2-12(3-\alpha)x_2+3}{4(\alpha-1)}\;\mbox{,}$$
      It follows that
      \begin{eqnarray*}
      x_1 & = & \frac{-\alpha^2-2\sqrt{\alpha^4-8 \alpha^3+23 \alpha^2-22 \alpha+10}+10 \alpha-5}{6 (\alpha+1)^2} \\
      x_2 & = & \frac{2 \alpha^2+\sqrt{\alpha^4-8 \alpha^3+23 \alpha^2-22 \alpha+10}-2 \alpha+4}{3 (\alpha+1)^2}\;\mbox{,}
      \end{eqnarray*}
      and the value of (\ref{eq-B-7A}) yet again smaller than (\ref{eq-B-M}).
\end{itemize}

It remains to consider the cases when none of the values of $x_1$, $x_2$ and $x_3$
is equal to $\frac{1}{\alpha+1}$ or $1/2$.
\begin{itemize}
\item {\bf Both $x_1$ and $x_2$ belong to $\left(0,\frac{1}{\alpha+1}\right)$, and
           $x_3$ belongs to $\left(\frac{1}{\alpha+1},1/2\right)$.}
      It follows that $x_1=x_2$ and that
      $$6x_1-3(\alpha+3)x_1^2=\frac{3(3-\alpha)(\alpha+5)x_3^2-12(3-\alpha)x_3+3}{4(\alpha-1)}\;\mbox{.}$$
      We get that either $x_1=x_2=1/4$ and $x_3=1/2$, which does not meet the description of the case, or
      $$x_1=x_2=\frac{-\alpha^2+2\alpha+4}{12}\mbox{ and }x_3=\frac{\alpha^2-2\alpha+2}{6}\;\mbox{.}$$
      The latter is the solution given in the statement of the proposition.
\item {\bf Both $x_1$ and $x_2$ belong to $\left(0,\frac{1}{\alpha+1}\right)$, and
           $x_3$ belongs to $\left(1/2,1\right)$.}
      We get that $x_1=x_2$ and that
      $$6x_1-3(\alpha+3)x_1^2=3x_3-\frac{3}{4}(\alpha+3)x_3^2\;\mbox{.}$$
      However, the only solution $x_1=x_2=1/4$ and $x_3=1/2$ does not meet the case description.
\item {\bf The value of $x_1$ belongs to $\left(0,\frac{1}{\alpha+1}\right)$, and
           both $x_2$ and $x_3$ belong to $\left(\frac{1}{\alpha+1},1/2\right)$.}
      It follows that $x_2=x_3$ and that
      $$6x_1-3(\alpha+3)x_1^2=\frac{3(3-\alpha)(\alpha+5)x_3^2-12(3-\alpha)x_3+3}{4(\alpha-1)}\;\mbox{.}$$
      We get that
      \begin{eqnarray*}
      x_1 & = & \frac{-\alpha^2+18 \alpha-13\pm2 \sqrt{4 \alpha^4-24 \alpha^3+53 \alpha^2-30 \alpha+1}}{3\cdot (5 \alpha^2+10 \alpha-11)} \\
      x_3 & = & \frac{8 \alpha^2+6 \alpha-10\mp\sqrt{4 \alpha^4-24 \alpha^3+53 \alpha^2-30 \alpha+1}}{3\cdot (5 \alpha^2+10 \alpha-11)}\;\mbox{,}
      \end{eqnarray*}
      which results in the value of (\ref{eq-B-7A}) to be smaller than (\ref{eq-B-M}).
\item {\bf The value of $x_1$ belongs to $\left(0,\frac{1}{\alpha+1}\right)$,
           the value of $x_2$ belongs to $\left(\frac{1}{\alpha+1},1/2\right)$, and
	   the value of $x_3$ belongs to $\left(1/2,1\right)$.}
      Using the method of the Lagrange multipliers, we get that
      $$6x_1-3(\alpha+3)x_1^2=\frac{3(3-\alpha)(\alpha+5)x_2^2-12(3-\alpha)x_2+3}{4(\alpha-1)}=3x_3-\frac{3}{4}(\alpha+3)x_3^2\;\mbox{.}$$
      Hence, $x_3=2x_1$, which implies that $x_1>1/4$.
      However, this is impossible since $x_2>x_1$ and
      the sum of $x_1$, $x_2$ and $x_3$ is equal to one.
\end{itemize}
The proof of the proposition is now finished.
Note that the feasible solution from the statement is unique up to a permutation of the values of the variables.
\end{proof}


\begin{thebibliography}{99}
\bibitem{bib-flag1}
R.~Baber:
{\em Tur\'an densities of hypercubes\/},
preprint available as arXiv: 1201.3587.
\bibitem{bib-flagrecent}
R.~Baber and J.~Talbot:
{\em A solution to the $2/3$ conjecture\/},
SIAM J. Discrete Math. {\bf 28} (2014), 756--766.
\bibitem{bib-flag2}
R.~Baber and J.~Talbot:
{\em Hypergraphs do jump\/},
Combin. Probab. Comput. {\bf 20} (2011), 161--171.
\bibitem{bib-flag3}
J.~Balogh, P.~Hu, B.~Lidick\'y and H.~Liu:
{\em Upper bounds on the size of 4- and 6-cycle-free subgraphs of the hypercube\/},
European J. Combin. {\bf 35} (2014), 75--85.
\bibitem{bib-borgs10+}
C.~Borgs, J.T.~Chayes and L.~Lov{\'a}sz:
{\em Moments of two-variable functions and the uniqueness of graph limits\/},
Geom. Funct. Anal. {\bf 19} (2010), 1597--1619.
\bibitem{bib-erdos62}
P.~Erd\H os:
{\em On the number of complete subgraphs contained in certain graphs\/},
Publ.~Math.~Inst.~Hung.~Acad.~Sci.~VII Ser.~A {\bf 3} (1962), 459--464.
\bibitem{bib-fisher89}
D.~Fisher:
{\em Lower bounds on the number of triangles in a graph\/}
J. Graph Theory {\bf 13} (1989), 505--512.
\bibitem{bib-franek93+}
F.~Franek, V.~R\"odl:
{\em 2-Colorings of complete graphs with a small number of monochromatic $K_4$ subgraphs\/},
Discrete Math. {\bf 114} (1993), 199--203.
\bibitem{bib-goodman59}
A.~W.~Goodman:
{\em On sets of acquaintances and strangers at any party\/},
Amer.~Math.~Monthly {\bf 66} (1959), 778--783.
\bibitem{bib-flag4}
A.~Grzesik:
{\em On the maximum number of five-cycles in a triangle-free graph\/},
J. Combin. Theory Ser. B {\bf 102} (2012), 1061--1066.
\bibitem{bib-flag5}
H.~Hatami, J.~Hladk\'y, D.~Kr\'al', S.~Norine and A.~Razborov:
{\em Non-three-colorable common graphs exist\/},
Combin. Probab. Comput. {\bf 21} (2012), 734--742.
\bibitem{bib-flag6}
H.~Hatami, J.~Hladk\'y, D.~Kr\'al', S.~Norine and A.~Razborov:
{\em On the number of pentagons in triangle-free graphs\/},
J. Combin. Theory Ser. A {\bf 120} (2013), 722--732.
\bibitem{bib-huang14+}
H.~Huang, N.~Linial, H.~Naves, Y.~Peled and B.~Sudakov:
{\em On the 3-local profiles of graphs\/},
J. Graph Theory {\bf 76} (2014), 236--248.
\bibitem{bib-huang+}
H.~Huang, N.~Linial, H.~Naves, Y.~Peled and B.~Sudakov:
{\em On the densities of cliques and independent sets in graphs\/},
Combinatorica, to appear.
\bibitem{bib-katona68}
G.~Katona:
{\em A theorem of finite sets},
in: P.~Erd\H os and G.~Katona (eds.): Theory of Graphs, Akad\'emiai Kiad\'o and Academic Press, 1968.
\bibitem{bib-flag7}
D.~Kr\'al', C.-H.~Liu, J.-S.~Sereni, P.~Whalen and Z.~Yilma:
{\em A new bound for the 2/3 conjecture\/},
Combin. Probab. Comput. {\bf 22} (2013), 384--393.
\bibitem{bib-flag8}
D.~Kr\'al', L.~Mach and J.-S.~Sereni:
{\em A new lower bound based on Gromov's method of selecting heavily covered points\/},
Discrete Comput. Geom. {\bf 48} (2012), 487--498.
\bibitem{bib-kruskal63}
J.~B.~Kruskal:
{\em The number of simplices in a complex},
in: R.~Bellman (ed.): Mathematical Optimization Techniques, University of California Press, 1963.
\bibitem{bib-lo12}
A.~Lo:
{\em Cliques in graphs with bounded minimum degree\/},
Combin. Probab. Comput. {\bf 21} (2012), 457--482.
\bibitem{bib-lovasz-book}
L.~Lov\'asz:
{\em Large networks and graph limits\/},
AMS, Providence, RI, 2012.
\bibitem{bib-lovasz06+}
L.~Lov{\'a}sz and B.~Szegedy:
{\em Limits of dense graph sequences\/},
J. Combin. Theory Ser. B {\bf 96} (2006), 933--957.
\bibitem{bib-nikiforov11}
V.~Nikiforov:
{\em The number of cliques in graphs of given order and size\/},
Trans.~Amer.~Math.~Soc. {\bf 363} (2011), 1599--1618.
\bibitem{bib-pikhurko16+}
O.~Pikhurko and A.~Razborov:
{\em Asymptotic structure of graphs with the minimum number of triangles\/},
Combin.~Probab.~Comput., to appear.
\bibitem{bib-flag9}
O.~Pikhurko and E.R.~Vaughan:
{\em Minimum number of k-cliques in graphs with bounded independence number\/},
Combin. Probab. Comput. {\bf 22} (2013), 910--934.
\bibitem{bib-razborov07}
A.~Razborov:
{\em Flag algebras\/},
J.~Symbolic Logic {\bf 72} (2007), 1239--1282.
\bibitem{bib-flag12}
A.~Razborov:
{\em On 3-hypergraphs with forbidden 4-vertex configurations\/},
SIAM J. Discrete Math. {\bf 24} (2010), 946--963.
\bibitem{bib-flag11}
A.~Razborov:
{\em On the minimal density of triangles in graphs\/},
Combin.~Probab.~Comput. {\bf 17} (2008), 603--618.
\bibitem{bib-reiher16}
C.~Reiher:
{\em The clique density theorem\/},
Ann.~of Math. {\bf 184} (2016), 683--707.
\bibitem{bib-sperfeld}
K.~Sperfeld:
{\em On the minimal monochromatic $K_4$-density\/},
preprint available as arXiv:1106.1030.
\bibitem{bib-thomason89}
A.~Thomason:
{\em A disproof of a conjecture of Erd\H os in Ramsey theory\/},
J.~London Math.~Soc. {\bf 39} (1989), 246--255.
\bibitem{bib-thomason97}
A.~Thomason:
{\em Graph products and monochromatic multiplicities\/},
Combinatorica {\bf 17} (1997), 125--134.
\bibitem{bib-turan41}
P.~Tur\'an:
{\em On an extremal problem in graph theory\/} (in Hungarian),
Matematikai \'es Fizikai Lapok {\bf 48} (1941), 436--452.
\bibitem{bib-wolf10}
J.~Wolf:
{\em The minimum number of monochromatic 4-term progressions in $\ZZ_p$\/},
J.~Comb. {\bf 1} (2010), 53--68.
\end{thebibliography}
\end{document}